\documentclass[leqno,12pt,a4paper]{article}
\usepackage[dvips]{graphicx}
\usepackage{amssymb}
\usepackage{color}
\usepackage{theorem}
\newcommand{\pref}[1]{(\ref{#1})}
\newcommand{\be}{\begin{equation}}
\newcommand{\ee}{\end{equation}}
\newcommand{\qed}{{\unskip\nobreak\hfil\penalty50\quad\null\nobreak\hfil
	$\square$\parfillskip0pt\finalhyphendemerits0\par\medskip}}
\newcommand{\zume}{\!\!\!}

\newcommand{\vece}{\mbox{\boldmath $ e $}}

\newcommand{\veco}{\mbox{\boldmath $ o $}}

\newcommand{\vecu}{\mbox{\boldmath $ u $}}

\newcommand{\vecx}{\mbox{\boldmath $ x $}}

\newtheorem{thm}{Theorem}[section]
\newtheorem{prop}{Proposition}[section]
\newtheorem{lem}{Lemma}[section]
\newtheorem{cor}{Corollary}[section]

\newtheorem{fact}{Fact}[section]

\theorembodyfont{\rmfamily}
\newtheorem{rem}{Remark}[section]
\newtheorem{proof}{\normalfont\itshape Proof.}

\title{On the global existence of generalized rotational hypersurfaces with prescribed mean curvature in the Euclidean spaces,
II}
\author{Takeyuki Nagasawa}
\date{}

\begin{document}
\maketitle
\begin{abstract}
In the previous paper,
it has been proved that the generalized rotational hypersurfaces of $ O( n-1 ) $-type and $ O ( \ell + 1 ) \times O( m+1 ) $-type,
for which the mean curvature is any prescribed continuous function.
This paper is a sequel,
and a similar existence result is shown for any type.
\end{abstract}
\section{Introduction}
\par
{Let $ H $ be any continuous function on the real line.} 
The purpose of this work is to construct generalized rotational hypersurfaces with any prescribed continuous mean curvature $ H $.
A generalized rotational hypersurface $ M $ in  $ n $-dimensional Euclidean space $\mathbb{R}^{n}$, for $ n \geqq 3 $, is defined via a compact Lie group $ G $ and its representation to $\mathbb{R}^{n}$,
{\em i.e.},
$ M $ is invariant under an isometric transformation group $( G ,\mathbb{R}^n ) $ with codimension two principal orbit type.
Such transformation groups $ (G , \mathbb{R}^n )$ are well  known and they were classified by Hsiang \cite{hsjdg} into five types:
\begin{description}
\item[{\rm Type I:}]	$ ( G , \mathbb{R}^n ) = ( O( n-1 ) , \mathbb{R}^n ) $.
\item[{\rm Type II:}]	$ ( G , \mathbb{R}^n ) = ( O( \ell + 1 ) \times O( m + 1 ) , \mathbb{R}^{ \ell + m + 2 } ) $.
\item[{\rm Type III:}]	$ ( G , \mathbb{R}^n ) = ( SO(3) , \mathbb{R}^5 ) $,
$ ( SU(3) , \mathbb{R}^8 ) $,
$ ( Sp (3) , \mathbb{R}^{14} ) $,
$ ( F_4 , \mathbb{R}^{26} ) $.
\item[{\rm Type IV:}]	$ ( G , \mathbb{R}^n ) = ( SO(5), \mathbb{R}^{10} ) $,
$ ( U(5) , \mathbb{R}^{20} ) $,
$ ( U(1) \times Spin (10) , \mathbb{R}^{32} ) $,\\[1mm] 
$ ( SO (2) \times SO (m) , \mathbb{R}^{2m} ) $,
$ ( S ( U(2) \times  U(m) ) , \mathbb{R}^{4m} ) $,
$ ( Sp (2) \times Sp (m) , \mathbb{R}^{8m} ) ) $.
\item[{\rm Type V:}]	$ ( G , \mathbb{R}^n ) = ( SO (4) , \mathbb{R}^8 ) $,
$ ( G_2 , \mathbb{R}^{14} ) $.
\end{description}
\par
 Hypersurfaces of Type I were constructed by Kenmotsu \cite{K} in the case $ n = 3 $ when $ H $ is a continuous function,
and   by Dorfmeister-Kenmotsu \cite{DK} for $ n \geqq 4 $, in the case when $ H $ is an analytic function.
In the previous paper \cite{KenNaga} it is shown the general existence of hypersurfaces of Type I and Type II, with  any continuous mean curvature $ H $.
In this paper, we get a similar result in all of the five Types I--V.
\par
Let $ ( x(s) , y(s) ) $ be the generating curve with archlength parameter $ s $:
\be
	( x^\prime (s) )^2 + ( y^\prime (s) )^2 = 1 .
	\label{archlength}
\ee
For each Type, another equation determines the generalized rotational hypersurfaces with mean curvature $ H $; these can be found in \cite{hsjdg}. We shall describe the equations uniformly by introducing the following notation:
\[
	\vecx (s) = \, ^t \! ( x (s) , y (s) ) ,
	\quad
	\vecx^\prime (s) = \, ^t \! ( x^\prime (s) , y^\prime (s) ) ,
	\quad
	\vecx^{\prime\prime} (s)^\bot = \, ^t \! ( - y^{ \prime \prime } (s) , x^{ \prime \prime } (s) )
	,
\]
and
\[
	\vece ( \phi ) = \, ^t \! ( \cos \phi , \sin \phi ) ,
	\quad
	\vece ( \phi )^\bot = \, ^t \! ( - \sin \phi , \cos \phi ) ,
\]

\begin{fact}
For each type of generalized rotational hypersurface, there exist a finite set $ J $,
angles $ \phi_j \in ( - \frac \pi 2 , \frac \pi 2 ] $,
and natural numbers $ n_j \in \mathbb{N} $ for $ j \in J $, such that the equation of the generalized rotational hypersurface can be described by
\be
	\vecx^{ \prime \prime } (s)^\bot \cdot \vecx^\prime (s)
	+
	\sum_{ j \in J }
	\frac { n_j \vece ( \phi_j ) \cdot \vecx^\prime (s) }
	{ \vece ( \phi_j )^\bot \cdot \vecx (s) }
	=
	( n - 1 ) H (s)
	,
	\quad
	\| \vecx^\prime (s) \|^2 = 1 ,
	\label{equationx}
\ee
where
\[
	\sum_{ j \in J } n_j = n - 2
\]
is the number of principal curvatures, not including the curvature of the generating curve.
For each type above,
the set $ J $, the 
angles $ \phi_j $,
and the natural numbers $ n_j $ for $ j \in J$ are given as follows:
\begin{description}
\item[{\rm Type I:}]
	$ J = \{ 0 \} $,
	$ \phi_0 = 0 $,
	$ n_0 = n - 2 $.
\item[{\rm Type II:}]
	$ J = \{ 0 , 1 \} $,
	$ \phi_1 = \frac \pi 2 $,
	$ \phi_0 = 0 $,
	$ n_0 = m $,
	$ n_1 = \ell $.
\item[{\rm Type III:}]
	$ J = \{ - 1 , 0 , 1 \} $
	$ \phi_1 = \frac \pi 3 $,
	$ \phi_j = j \phi_1 $.
	And $ n_j \equiv 1 $,
	$ 2 $,
	$ 4 $,
	or $ 8 $ for $ ( SO(3) , \mathbb{R}^5 ) $,
$ ( SU(3) , \mathbb{R}^8 ) $,
$ ( Sp (3) , \mathbb{R}^{14} ) $,
or $ ( F_4 , \mathbb{R}^{26} ) $ respectively.
\item[{\rm Type IV:}]
	$ J = \{ - 1 , 0 , 1 , 2 \} $,
	$ \phi_1 = \frac \pi 4 $,
	$ \phi_j = j \phi_1 $.
	And $ n_{ \pm 1 } = \ell $,
	$ n_0 = n_2 = k $,
	where
	$ ( k, \ell ) = ( 2,2 ) $,
	$ ( 5,4 ) $,
	$ ( 9,6 ) $,
	$ ( m-2 ,1 ) $,
	$ ( 2m-3 , 2 ) $,
	$ ( 4m-5 , 4 ) $ for $ ( SO(5), \mathbb{R}^{10} ) $,
	$ ( U(5) , \mathbb{R}^{20} ) $,
	$ ( U(1) \times Sp (10) , \mathbb{R}^{32} ) $,
	$ ( SO (2) \times SO (m) , \mathbb{R}^{2m} ) $,
	$ ( S ( U(2) \times  U(m) ) , \mathbb{R}^{4m} ) $,
	$ ( Sp (2) \times Sp (m) , \mathbb{R}^{8m} ) ) $ respectively.
\item[{\rm Type V:}]
	$ J = \{ -2 , -1 , 0 , 1 , 2 , 3 \} $,
	$ \phi_1 = \frac \pi 6 $,
	$ \phi_j = j \phi_1 $.
	And $ n_j \equiv 1 $,
	or $ 2 $ for $ ( SO (4) , \mathbb{R}^8 ) $,
	or $ ( G_2 , \mathbb{R}^{14} ) $ respectively.
\end{description}
\end{fact}
\par
We get these by direct calculations.
For example the equation for Type II,
see \cite{KenNaga}, 
is:
\[
	x^{\prime\prime} (s) y^\prime (s) - y^{\prime\prime} (s) x^\prime (s) 
	- \frac { \ell y^\prime (s) } { x (s) }
	+ \frac { m x^\prime (s) } { y (s) }
	=
	( n - 1 ) H (s) ,
\]
together with equation \pref{archlength}. 
Each term in the left-hand side of the above equation is
\[
	x^{\prime\prime} (s) y^\prime (s) - y^{\prime\prime} (s) x^\prime (s)
	= 
	\vecx^{ \prime \prime } (s)^\bot \cdot \vecx^\prime (s)
	,
\]
\[
	- \frac { \ell y^\prime (s) } { x (s) }
	=
	\frac { \ell \, ^t \! ( 0 , 1 ) \cdot \, ^t \! ( x^\prime (s) , y^\prime (s) ) } { \, ^t \! ( -1 , 0 ) \cdot \, ^t \! ( x (s) , y (s)  ) }
	=
	\frac { n_1 \vece ( \phi_1 ) \cdot \vecx^\prime (s) } { \vece ( \phi_1 )^\bot \cdot \vecx (s) }
	,
\]
\[
	\frac { m x^\prime (s) } { y (s) }
	=
	\frac { m \, ^t \! ( 1 , 0 ) \cdot \, ^t \! ( x^\prime (s) , y^\prime (s) ) } { \, ^t \! ( 0,1 ) \cdot \, ^t \! ( x (s) , y (s) ) }
	=
	\frac { n_0 \vece ( \phi_0 ) \cdot \vecx^\prime } { \vece ( \phi_0 )^\bot \cdot \vecx (s) }
	,
\]
and
\[
	\sum_{ j \in J } n_j
	=
	\ell + m
	=
	n - 2 .
\]
\par
Our main result is:
\begin{thm}
Let $ H $ be a continuous function on $ \mathbb{R} $.
Put
\[
	S = \{ \vecx \in \mathbb{R}^2 \, | \, \vece ( \phi_j )^\bot \cdot \vecx = 0 \mbox{ for some } j \in J \} .
\]
For any $ \vecx_0 \not \in S $ and $ s_0 \in \mathbb{R} $,
there exists a solution $ \vecx(s) $ to {\rm \pref{equationx}} on $ \mathbb{R} $ satisfying $ \vecx( s_0 ) = \vecx_0 $.
\label{thm1.1}
\end{thm}
\par
Our equation is singular on the set $ S $.
Since $ \vecx ( s_0 ) \not \in S $,
then it is simple to construct a solution near $ s = s_0 $,
and we can extend the solution as long as $ \vecx (s) \not \in S $.
To extend the solution,
a problem happens when a solution approaches to $ S $ as $ s \to s_\ast $ for some $ s_\ast \in \mathbb{R} $.
It is a non trivial  fact that the solution can be extended beyond $ s = s_\ast $.
We shall study the asymptotic behavior of $ \vecx^\prime ( s ) $ as $ s \to s_\ast $,
and in particular the existence of the limit $ \displaystyle{ \lim_{ s \to s_\ast } \vecx^\prime ( s ) } $,
say $ \vecx^\prime_\ast $.
Furthermore, we shall construct solutions beyond $ s_\ast $ with $ \vecx^\prime (s_\ast) = \vecx^\prime_\ast $.
\par
By a formal blow-up argument we can evaluate the limit $ \vecx_\ast^\prime $.
For simplicity we shall assume $ s_\ast = 0 $.
Let the generating curve be in the sector
\be
	S_i = \left\{
	\begin{array}{ll}
	\{ \vecx \in \mathbb{R}^2 \, | \,
	\vece ( 0 )^\bot \cdot \vecx > 0 \}
	= \\ \hfill 
	= \{ ( x,y ) \in \mathbb{R}^2 \, | \, y > 0 \}
	& \mbox{for Type I},
	\\[1mm]
	\{ \vecx \in \mathbb{R}^2 \, | \, \vece( \phi_{ i+1 } )^\bot \cdot \vecx < 0 < \vece( \phi_i )^\bot \cdot \vecx \}
	& \mbox{for Types II--V}
	\end{array}
	\right.
	\label{sector-i}
\ee
for $ i = 0 $ (Type I),
or for some $ i \in J \setminus \{ \max J \} $ (Types II--V).
There are three cases:
(i) $ \displaystyle{ \lim_{ s \to 0 } \vecx ( s ) = \vecx_\ast \ne \veco } $,
$ \vece ( \phi_i )^\bot \cdot \vecx_\ast = 0 $ (the condition $ \vecx_\ast \ne \veco $ can be removed for Type I by a translation);
(i)$^\prime $ $ \displaystyle{ \lim_{ s \to 0 } \vecx ( s ) = \vecx_\ast \ne \veco } $,
$ \vece ( \phi_{i+1} )^\bot \cdot \vecx_\ast = 0 $;
and (ii) $ \displaystyle{ \lim_{ s \to 0 } \vecx ( s ) = \veco } $ for Types II--V.
Since the argument for (i)$^\prime $ is similar to that for (i),
we shall consider cases (i) and (ii) only.
\par
For case (i),
we assume that there exists the limit $ \displaystyle{ \lim_{ s \to + 0 } \vecx^\prime ( s ) = \vece (\theta_\ast ) } $ and that $ \vecx^{ \prime \prime } (s) $ is bounded.
Put
\[
	\vecx_\lambda ( s ) = \lambda^{-1} ( \vecx ( \lambda s ) - \vecx_\ast )
\]
for $ \lambda > 0 $. 
Then it is easy to see that
\[
	\vecx_\lambda^{ \prime \prime } ( s )^\bot \cdot \vecx_\lambda^\prime (s)
	+
	\sum_{ j \in J } \frac
	{ \lambda n_j \vece ( \phi_j ) \cdot \vecx_\lambda^\prime ( s ) }
	{ \vece ( \phi_j )^\bot \cdot \left( \lambda \vecx_\lambda (s) - \vecx_\ast \right) }
	=
	( n -  1 ) \lambda H ( \lambda s ) .
\]
Since we have
\[
	\lim_{ \lambda \to + 0 } \vecx_\lambda (s)
	= s \vecx^\prime (0)
	= s \vece ( \theta_\ast )
	,
	\quad
	\lim_{ \lambda \to + 0 } \vecx_\lambda^\prime (s)
	= \vecx^\prime (0) = \vece ( \theta_\ast ) ,
	\quad
	\lim_{ \lambda \to + 0 } \vecx_\lambda^{ \prime \prime } (s)
	= \veco ,
\]
we get
\[
	\vece ( \phi_i ) \cdot \vece ( \theta_\ast ) = 0
\]
by letting $ \lambda \to + 0 $.
Consequently,
$ \theta_\ast = \phi_i + \frac \pi 2 $.
We can obtain a similar result letting $ s \to - 0 $.
Thus the generating curve touches perpendicularly the boundary of sector $ S_i $.
\par
For case (ii),
we assume the existence of $ \displaystyle{ \lim_{ s \to + 0 } \vecx^\prime ( s ) = \vece (\theta_\ast ) } $,
$ \theta_\ast \ne \phi_j $ for $ j \in J $,
and the boundedness of $ \vecx^{ \prime \prime } (s) $.
Putting
\[
	\vecx_\lambda (s) = \lambda^{-1} \vecx ( \lambda s ) ,
\]
we have
\[
	\vecx_\lambda^{ \prime \prime } ( s )^\bot \cdot \vecx_\lambda^\prime (s)
	+
	\sum_{ j \in J } \frac
	{ n_j \vece ( \phi_j ) \cdot \vecx_\lambda^\prime ( s ) }
	{ \vece ( \phi_j )^\bot \cdot \vecx_\lambda (s) }
	=
	( n -  1 ) \lambda H ( \lambda s ) ,
\]
and
\[
	\sum_{ j \in J } n_j \cot ( \theta_\ast - \phi_j ) = 0
\]
by $ \lambda \to + 0 $.
Put
\be
	A( \theta ) = \sum_{ j \in J } n_j \cot ( \theta - \phi_j ) .
	\label{functionA}
\ee
Since $ A ( \cdot ) $ is monotone decreasing on each interval $ ( \phi_i , \phi_{ i+1 } ) $,
and since
\[
	\lim_{ \theta \to \phi_i + 0 } A( \theta ) = \infty ,
	\quad
	\lim_{ \theta \to \phi_{ i+1 } - 0 } A( \theta ) = - \infty
\]
there exists a unique $ \theta_i $ on each interval $ ( \phi_i , \phi_{ i+1 } ) $ such that $ A( \theta_i ) = 0 $.
Thus the generating curve approaches to the origin with angle $ \theta_i $.
\par
In the above argument we assume the existence of $ \displaystyle{ \lim_{ s \to 0 } \vecx^\prime (s) } $,
the boundedness $ \vecx^{ \prime \prime } (s) $ and so on.
In the following sections,
we shall prove the asymptotic behavior as above without these assumptions,
and shall show the existence of solutions of \pref{equationx} with the initial value
\be
	\vece ( \phi_i )^\bot \cdot \vecx (0) = 0 ,
	\quad
	\vecx (0) \ne \veco ,
	\quad
	\vecx^\prime (0) = \vece \left( \phi_i + \frac \pi 2 \right)
	\label{initialcondition-i}
\ee
for all Types I--V,
or
\be
	\vecx (0) = 0 ,
	\quad
	\vecx^\prime (0) = \vece( \theta_i )
	\label{initialcondition-ii}
\ee
for Types II--V.
\begin{rem}
By direct calculation,
we get the explicit values of each $ \theta_i $:
\begin{description}
\item[{\rm Type II:}]
	$ \theta_0 = \arctan\sqrt{ \frac { n_0 } { n_1 } } $.
\item[{\rm Type III:}]
	$ \theta_i = \frac 12 ( \phi_i + \phi_{ i+1 } ) $.
\item[{\rm Type IV:}]
	$ \theta_{\pm 1 } = - \frac 12 \arctan \sqrt{ \frac k \ell } $,
	$ \theta_0 = \frac 12 \arctan \sqrt{ \frac k \ell } $.
\item[{\rm Type V:}]
	$ \theta_i = \frac 12 ( \phi_i + \phi_{ i+1 } ) $.
\end{description}
For the calculations, see Fact \ref{thetai} in the Appendix.
\end{rem}
\par
We shall discuss the following two cases in \S\S~3--4 respectively:
\begin{description}
\item[{\rm Case (i)}]
The asymptotic behavior of $ \vecx^\prime (s) $ when $ \displaystyle{ \lim_{ s \to 0 } \vecx (s) = \vecx_\ast \ne \veco } $,
and the solvability of the initial-value problem \pref{equationx} and \pref{initialcondition-i},
in Propositions \ref{prop3.1}--\ref{prop3.2}.
\item[{\rm Case (ii)}]
The asymptotic behavior of $ \vecx^\prime (s) $ when $ \displaystyle{ \lim_{ s \to 0 } \vecx (s) = \veco } $,
and the solvability of the initial-value problem \pref{equationx} and \pref{initialcondition-ii},
in Propositions \ref{prop4.1}--\ref{prop4.2}.
\end{description}
\par
Theorem \ref{thm1.1} follows then from these Propositions.
For all types in  Case (i) and for Types II--III in  Case (ii),
the derivations for setting the problem are more complicated but similar to those in \cite{KenNaga}.
Therefore, we shall present the setting in detail, but  some estimates  just briefly.
For the Types IV--V in the case (ii),
we need here one extra procedure in comparison to \cite{KenNaga}.
We can prove our results without the extra procedure for Types II--III as \cite{KenNaga},
but this extra procedure is applicable  for all types, and
in this sense the proof shall be universal.
\section{A transformation}
The following is a useful transformation.
We define the matrix $ R ( \psi ) $ and a vector $ \vecu = \, ^t ( u,v ) $ by
\par
\[
	R( \psi ) =
	\left(
	\begin{array}{rr}
	\cos \psi & - \sin \psi \\
	\sin \psi & \cos \psi
	\end{array}
	\right)
	,
	\quad
	\vecu = \, ^t \! ( u,v ) 
	= R( \psi ) \vecx .
\]
It is easy to see that when $ \vecx (s) $ satisfies \pref{equationx},
the new unknown vector function $ \vecu (s) = ( u(s) , v(s) ) $ satisfies
\[
	\vecu^{ \prime \prime } (s)^\bot \cdot \vecu^\prime (s)
	+
	\sum_{ j \in J }
	\frac { n_j \vece ( \phi_j + \psi ) \cdot \vecu^\prime (s) }
	{ \vece ( \phi_j + \psi )^\bot \cdot \vecu (s) }
	=
	( n - 1 ) H (s)
	,
	\quad
	\| \vecu^\prime (s) \|^2 = 1 .
\]

\par
Assume that the generating curve $ \vecx (s) $ is in the sector $ S_i $ defined by \pref{sector-i}.
We transform $ \vecx $ to $ \vecu $ with $ \psi = - \phi_i $,
then $ \vecu (s) $ is in the sector $ S_0 $ in $ uv $-plane.
\section{Case (i)}
\par
First we show
\begin{prop}
Let the generating curve $ \vecx (s) $ be in the sector $ S_i $,
and assume that
\[
	\lim_{ s \to 0 }
	\vece ( \phi_i )^\bot \cdot \vecx (s) = 0 ,
	\quad
	\lim_{ s \to 0 }
	\vecx (s) = \vecx_0 \ne 0 .
\]
Then there exists the limit of $ \vecx^\prime (s) $ as $ s \to 0 $ and
\[
	\lim_{ s \to 0 } \vece ( \phi_i ) \cdot \vecx^\prime (s) = 0 .
\]
\label{prop3.1}
\end{prop}
\begin{proof}
As stated in \S~2,
we transform $ \vecx (s) $ to $ \vecu (s) $ with $ \psi = - \phi_i $.
Then $ \vecu (s) = ( u(s) , v(s) ) $ satisfies
\be
	u^{\prime\prime} v^\prime - v^{\prime\prime} u^\prime
	+
	\frac { n_i u^\prime } v
	+
	\sum_{ j \ne i }
	\frac { n_j \vece ( \phi_j - \phi_i ) \cdot \vecu^\prime }
	{ \vece ( \phi_j - \phi_i )^\bot \cdot \vecu }
	=
	( n - 1 ) H
	,
	\quad
	\| \vecu^\prime \|^2 = 1
	.
	\label{equationu1}
\ee
The assumption on $ \displaystyle{ \lim_{ s \to 0 } \vecx (s) } $ is written as
\be
	\lim_{ s \to 0 } u(s) > 0 ,
	\quad
	\lim_{ s \to 0 } v(s) = + 0
	\label{assumptiononuv}
\ee
in terms of $ u(s) $ and $ v(s) $.
What we want to show is
\[
	\lim_{ s \to 0 } u^\prime (s) = 0 .
\]
\par
Multiplying both sides of the first equation of \pref{equationu1} by $ v^{ n_i } v^\prime $,
and using the second relation,
we have
\[
	\left( v^{ n_i } u^\prime \right)^\prime
	=
	\left\{
	( n - 1 ) H
	-
	\sum_{ j \ne i }
	\frac
	{ n_j \vece ( \phi_j - \phi_i ) \cdot \vecu^\prime }
	{ \vece ( \phi_j - \phi_i )^\bot \cdot \vecu }
	\right\}
	v^{ n_i } v^\prime
	.
\]
Taking \pref{assumptiononuv} into account,
we get
\[
	u^\prime (s)
	=
	\frac 1 { v^{ n_i } (s) }
	\int_0^s
	\left\{
	( n - 1 ) H(t)
	-
	\sum_{ j \ne i }
	\frac
	{ n_j \vece ( \phi_j - \phi_i ) \cdot \vecu^\prime (t) }
	{ \vece ( \phi_j - \phi_i )^\bot \cdot \vecu (t) }
	\right\}
	v^{ n_i } (t) v^\prime (t)
	\, dt
\]
by integration of the equation of $ u(s) $.
\par
Now we show $ v^\prime (s) \ne 0 $ near $ s = 0 $.
Assume that there exists a sequence $ \{ s_k \} $ such that
$ v^\prime ( s_k ) = 0 $,
$ \displaystyle{ \lim_{ k \to \infty } s_k = 0 } $.
Inserting $ s = s_k $ into $ \left( u^\prime \right)^2 + \left( v^\prime \right)^2 \equiv 1 $ and 
$ u^\prime u^{ \prime \prime } + v^\prime v^{ \prime \prime } \equiv 0 $,
we have
\[
	u^\prime ( s_k ) = \pm 1 ,
	\quad
	u^{ \prime \prime } ( s_k ) = 0
	.
\]
By evaluating the equation at $ s = s_k $, it follows that
\[
	v^{ \prime \prime } ( s_k )
	=
	\mp \left( 1 + \sum_{ j \in J } n_j \right) H( s_k )
	+ \frac { n_i } { v ( s_k ) }
	+
	\sum_{ j \ne i }
	\frac { n_j \cos ( \phi_j - \phi_i ) }
	{ \vece ( \phi_j - \phi_i )^\bot  \cdot \vecu ( s_k ) }
\]
It holds that
\[
	\lim_{ k \to \infty } \vece ( \phi_j - \phi_i )^\bot \cdot \vecu ( s_k )
	=
	- u(0) \sin ( \phi_j - \phi_i ) \ne 0
	\quad \mbox{for} \quad
	j \ne i .
\]
Therefore,
we have
\[
	\lim_{ k \to \infty } v^{ \prime \prime } ( s_k )
	=
	\lim_{ k \to \infty } \frac { n_i } { v( s_k ) }
	=
	\infty
	.
\]
Consequently,
$ v( s_k ) $'s are always local minimum values for large $ k $.
This contradicts the asumption that $ \displaystyle{ \lim_{ s \to 0 } v(s) = + 0 } $.
\par
Since $ v^\prime (s) \ne 0 $ near $ s = 0 $,
we can use L'Hospital's theorem to obtain
\[
	\begin{array}{rl}
	\displaystyle{
	\lim_{ s \to 0 } u^\prime (s)
	}
	= & \zume
	\displaystyle{
	\lim_{ s \to 0 }
	\frac
	{
	\displaystyle{
	\left\{
	\left( 1 + \sum_{ j \in J } n_j \right) H(s)
	-
	\sum_{ j \ne i }
	\frac
	{ n_j \vece ( \phi_j - \phi_i ) \cdot \vecu^\prime (s) }
	{ \vece ( \phi_j - \phi_i )^\bot \cdot \vecu (s) }
	\right\}
	v^{ n_i } (s) v^\prime (s)
	}}
	{ n_i v^{ n_i -1 } (s) v^\prime (s) }
	}
	\\
	= & \zume
	\displaystyle{
	\lim_{ s \to 0 }
	\left\{
	\left( 1 + \sum_{ j \in J } n_j \right) H(s)
	-
	\sum_{ j \ne i }
	\frac
	{ n_j \vece ( \phi_j - \phi_i ) \cdot \vecu^\prime (s) }
	{ \vece ( \phi_j - \phi_i )^\bot \cdot \vecu (s) }
	\right\}
	\frac { v (s) } { n_i }
	}
	\\
	= & \zume
	0
	.
	\end{array}
\]
Here we use
\[
	\lim_{ s \to 0 } \vece ( \phi_j - \phi_i )^\bot \cdot \vecu ( s )
	=
	- u(0) \sin ( \phi_j - \phi_i ) \ne 0
	\quad \mbox{for} \quad
	j \ne i .
\]
\qed
\end{proof}
\par
Next we prove the converse of Proposition \ref{prop3.1},
{\em i.e.},
the solvability of \pref{equationx} and \pref{initialcondition-i}.
The problem is equivalent to \pref{equationu1} and
\be
	u^\prime (0) = 0
	.
	\label{initialconditionu-i}
\ee
Since $ v^\prime (0)^2 = 1 $,
the map $ s \mapsto v $ is monotone near $ s = 0 $.
Therefore,
there exists the inverse function $ s = s(v) $.
Put
\[
	q = \frac { du } { dv } = \frac { u^\prime } { v^\prime }
\]
as a function of $ v $.
We divide both sides of the first equation of \pref{equationu1} by $ \left( v^\prime \right)^3 $.
Using $ \| \vecu^\prime (s) \| \equiv 1 $,
we get
\be
	\begin{array}{l}
	\displaystyle{
	\frac { dq } { dv } + \frac { n_i q  } v
	=
	- \frac { n_i q^3 } v
	}
	\\[3mm]
	\quad
	\displaystyle{
	+ \,
	\sum_{ j \ne i }
	\frac
	{
	n_j \left( q^2 + 1 \right)
	\left\{ q \cos ( \phi_j - \phi_i ) + \sin ( \phi_j - \phi_i ) \right\}
	}
	{ \displaystyle{
	\left( u(0) + \int_0^v q( \eta ) \, d \eta \right) \sin ( \phi_j - \phi_i )
	-
	v \cos ( \phi_j - \phi_i )
	}
	} }
	\\[4mm]
	\quad
	\displaystyle{
	+ \,
	( n - 1 )
	\left( q^2 + 1 \right)^{ \frac 32 } \tilde H
	}
	,
	\end{array}
	\label{diffeqq1}
\ee
where
\[
	\tilde H = \left( \mbox{\rm sgn} \ v^\prime \right) H .
\]
We multiply both sides by $ v^{ n_i } $ and integrate from $ 0 $ to $ v $.
Since
\[
	\lim_{ v \to 0 } q(v) = \lim_{ s \to 0 } \frac { u^\prime (s) } { v^\prime (s) } = 0
	,
\]
we obtain
\be
	q(v) = \frac 1 { v^{ n_i } } \int_0^v \varphi( q ) ( \eta ) \, d \eta
	,
	\label{inteqq1}
\ee
where
\[
	\varphi( q ) ( \eta )
	=
	\varphi_1 ( q ) ( \eta )
	+
	\varphi_2 ( q ) ( \eta )
	+
	\varphi_3 ( q ) ( \eta )
	,
\]
\[
	\varphi_1 ( q ) ( \eta )
	=
	- n_i q( \eta )^3 \eta^{ n_i - 1 }
	,
\]
\[
	\varphi_2 ( q ) ( \eta )
	=
	\sum_{ j \ne i }
	\frac
	{
	\left( q^2 + 1 \right)
	\left\{ q ( \eta ) \cos ( \phi_j - \phi_i ) + \sin ( \phi_j - \phi_i ) \right\}
	\eta^{ n_i }
	}
	{ \displaystyle{
	\left( u(0) + \int_0^\eta q( \zeta ) \, d \zeta \right) \sin ( \phi_j - \phi_i )
	-
	\eta \cos ( \phi_j - \phi_i )
	}
	}
	,
\]
\[
	\varphi_3 ( q ) ( \eta )
	=
	( n - 1 )
	\left( q( \eta )^2 + 1 \right)^{ \frac 32 }
	\tilde H( \eta ) \eta^{ n_i }
	.
\]
\par
Define the Banach space $ X_V $ and its bounded set $ X_{V,M} $ by
\[
	X_V
	=
	\left\{ f \in C( 0 , V ] \, | \,
	\| f \| < \infty \right\}
	,
	\quad
	\| f \| = \sup_{ v \in ( 0 , V ] } \left| \frac { f ( v ) } v \right|
	,
\]
\[
	X_{V,M} = \{ f \in X_V \, | \, \| f \| \leqq M \}
	.
\]
Using the boundedness of $ H $,
we can show that if $ M $ is large and if $ V $ is small,
the map
\[
	\Phi ( q ) ( v )
	=
	\frac 1 { v^{ n_i } } \int_0^v \varphi (q) ( \eta ) \, d \eta
\]
defined on $ X_{V,M} $ into itself and it is contraction.
Indeed we have
\[
	\left\| \frac 1 { v^{ n_i } } \int_0^v \phi_1 ( q ) ( \eta ) \, d \eta \right\|
	\leqq C M^3 V^2 ,
\]
\[
	\left\| \frac 1 { v^{ n_i } } \int_0^v \phi_2 ( q ) ( \eta ) \, d \eta \right\|
	\leqq C \left( 1 + M^3 V^2 \right) ,
\]
\[
	\left\| \frac 1 { v^{ n_i } } \int_0^v \phi_3 ( q ) ( \eta ) \, d \eta \right\|
	\leqq C \left( 1 + M^3 V^2 \right)
\]
for $ q \in X_{M,V} $;
\[
	\left\| \frac 1 { v^{ n_i } } \int_0^v
	\left( \phi_1 ( q_1 ) ( \eta ) - \phi_1 ( q_2 ) ( \eta ) \right)
	d \eta \right\|
	\leqq
	C M^2 V^2 \| q_1 - q_2 \|
	,
\]
\[
	\left\| \frac 1 { v^{ n_i } } \int_0^v
	\left( \phi_2 ( q_1 ) ( \eta ) - \phi_2 ( q_2 ) ( \eta ) \right)
	d \eta \right\|
	\leqq
	C \left( M^2 V^3 + V + M V^2 + V^2 \right) \| q_1 - q_2 \|
	,
\]
\[
	\left\| \frac 1 { v^{ n_i } } \int_0^v
	\left( \phi_2 ( q_1 ) ( \eta ) - \phi_2 ( q_2 ) ( \eta ) \right)
	d \eta \right\|
	\leqq
	C \left( M^2 V^3 + M V^2 \right) \| q_1 - q_2 \|
\]
for $ q_1 \in X_{M,V} $ and $ q_2 \in X_{M,V} $.
Since these estimates can be obtained in the same way as in \cite{KenNaga},
we omit details.
Hence,
there exists the unique fixed point of $ \Phi $ in $ X_{M,V} $,
which solves \pref{inteqq1}.
If $ H $ is continuous,
then it solves \pref{diffeqq1} satisfying $ q(0) = 0 $.
From this fact we get the solvability of the original problem, and the proof of:
\begin{prop}
Let $ H $ be continuous.
Then there exists a unique local solution $ \vecx $ to {\rm \pref{equationx}} and {\rm \pref{initialcondition-i}}.
\label{prop3.2}
\end{prop}
\section{Case (ii)}
Consider the equation for Types II--V.
\begin{prop}
Let the generating curve $ \vecx (s) $ be in the sector $ S_i $,
and assume that
\[
	\lim_{ s \to \pm 0 } \vecx (s) = \veco .
\]
Then there exists the limit of $ \vecx^\prime (s) $ as $ s \to \pm 0 $ and
\[
	\lim_{ s \to \pm 0 } \vecx^\prime (s) = \pm \vece ( \theta_i ) .
\]
Here $ \theta_i $ is the unique angle satisfying
\[
	\sum_{ j \in J } n_j \cot ( \theta_i - \phi_j ) = 0 ,
	\quad
	\phi_i < \theta_i < \phi_{ i+1 } .
\]
\label{prop4.1}
\end{prop}
\par
This proposition is proved by a series of Lemmas.
In what follows,
$ \vecx (s) $ satisfies the assumption in Proposition \ref{prop4.1}.
For simplicity we consider only the case $ s \to + 0 $,
and assume that $ \vecx (s) $ is defined on $ ( 0 , s_0 ] $.
As in the previous section,
we transform $ \vecx $ to $ \vecu $ with $ \psi = - \phi_i $.
Then it holds that
\[
	\left( v^{ n_i } u^\prime \right)^\prime
	=
	\left\{
	( n - 1 ) H
	-
	\sum_{ j \ne i }
	\frac
	{ n_j \vece ( \phi_j - \phi_i ) \cdot \vecu^\prime }
	{ \vece ( \phi_j - \phi_i )^\bot \cdot \vecu }
	\right\}
	v^{ n_i } v^\prime
	.
\]
We integrate this from $ s_0 $ to $ s \in ( 0 , s_0 ) $,
and get
\[
	\begin{array}{rl}
	v^{ n_i } (s) u^\prime (s)
	= & \zume
	\displaystyle{
	\int_{ s_0 }^s
	\left\{
	( n - 1 ) H(t)
	-
	\sum_{ j \ne i }
	\frac
	{ n_j \vece ( \phi_j - \phi_i ) \cdot \vecu^\prime (t) }
	{ \vece ( \phi_j - \phi_i )^\bot \cdot \vecu (t) }
	\right\}
	v^{ n_i } (t) v^\prime (t) \,
	dt
	}
	\\
	& \qquad \qquad \qquad \qquad \qquad \qquad \qquad \qquad
	+ \,
	v^{ n_i } ( s_0 ) u^\prime ( s_0 )
	.
	\end{array}
\]
Since the left-hand side tends to $ 0 $ as $ s \to + 0 $,
so does the right-hand side.
Hence,
we get
\be
	\begin{array}{rl}
	u^\prime (s)
	= & \zume 
	\displaystyle{
	\frac 1 { v^{ n_i } (s) }
	\left[
	\int_{ s_0 }^s
	\left\{
	( n - 1 ) H(t)
	-
	\sum_{ j \ne i }
	\frac
	{ n_j \vece ( \phi_j - \phi_i ) \cdot \vecu^\prime (t) }
	{ \vece ( \phi_j - \phi_i )^\bot \cdot \vecu (t) }
	\right\}
	v^{ n_i } (t) v^\prime (t) \,
	dt
	\right.
	}
	\\
	& \qquad \qquad \qquad \qquad \qquad \qquad \qquad \qquad \qquad \qquad
	\displaystyle{
	\left.
	+ \,
	v^{ n_i } ( s_0 ) u^\prime ( s_0 )
	\vphantom{ \left\{\left( \sum_{ j \in J } \right)\right\} }
	\right]
	}
	.
	\end{array}
	\label{uprimeinlem2.1}
\ee
Next we shall apply L'Hospital's theorem to this.
\begin{lem}
If the limit $ \vecu^\prime (s) $ as $ s \to + 0 $ exists,
then it holds that
\[
	\lim_{ s \to + 0 } \vecu^\prime (s) = \vece ( \theta_i - \phi_i ) .
\]
\label{iflimitexists}
\end{lem}
\begin{proof}
Since $ \vecu^\prime (s) $ is a unit vector,
so is its limit,
say $ \vece ( \psi_\ast )$.
Under the assumption we can apply L'Hospital's theorem to \pref{uprimeinlem2.1},
and get
\be
	\begin{array}{rl}
	\cos \psi_\ast
	= & \zume
	\displaystyle{
	\frac 1 { n_i }
	\lim_{ s \to +0 }
	\left\{
	( n - 1 ) H(s)
	-
	\sum_{ j \ne i }
	\frac
	{ n_j \vece ( \phi_j - \phi_i ) \cdot \vecu^\prime (s) }
	{ \vece ( \phi_j - \phi_i )^\bot \cdot \vecu (s) }
	\right\}
	v(s)
	}
	\\
	= & \zume
	\displaystyle{
	- \frac 1 { n_i }
	\lim_{ s \to +0 }
	\sum_{ j \ne i }
	\frac
	{ n_j \vece ( \phi_j - \phi_i ) \cdot \vece ( \psi_\ast ) v(s) }
	{ \vece ( \phi_j - \phi_i )^\bot \cdot \vecu (s) }
	}
	\\
	= & \zume
	\displaystyle{
	- \frac 1 { n_i }
	\lim_{ s \to +0 }
	\sum_{ j \ne i }
	\frac
	{ n_j \cos ( \phi_j - \phi_i - \psi_\ast ) v(s) }
	{ \vece ( \phi_j - \phi_i )^\bot \cdot \vecu (s) }
	}
	.
	\end{array}
	\label{thetaastinlem2.1}
\ee
Since $ \vecu \in S_0 $ in $ uv $-plane,
it holds that $ \psi_\ast \in [ 0 , \phi_{ i+1 } - \phi_i ] $.
\par
We will show $ \psi_\ast \in ( 0 , \phi_{ i+1 } - \phi_i ) $.
Assume $ \psi_\ast = 0 $,
and then
\[
	\vece ( \phi_j - \phi_i )^\bot \cdot \vece ( \psi_\ast )
	= \sin ( \phi_j - \phi_i ) \ne 0
\]
for $ j \ne i $.
By use of L'Hopital's theorem again we have
\[
	\lim_{ s \to +0 }
	\frac { v(s) } { \vece ( \phi_j - \phi_i )^\bot \cdot \vecu (s) }
	=
	\lim_{ s \to +0 }
	\frac { v^\prime (s) } { \vece ( \phi_j - \phi_i )^\bot \cdot \vecu^\prime (s) }
	=
	\frac { \sin \psi_\ast } { \sin ( \phi_j - \phi_i ) }
	= 0
	.
\]
Hence,
from \pref{thetaastinlem2.1} it follows that
\[
	\cos \psi_\ast
	=
	- \frac 1 { n_i } \times 0 = 0
	.
\]
This contradicts  the asumption $ \psi_\ast = 0 $.
We can show $ \psi_\ast \ne \phi_{ i+1 } - \phi_i $ in a similar argument to $ R ( - \phi_{ i+1 } ) \vecx $.
\par
We have already known $ \psi_\ast \in ( 0 , \phi_{ i+1 } - \phi_i ) $,
and therefore
\[
	\vece ( \phi_j - \phi_i )^\bot \cdot \vece ( \psi_\ast )
	=
	\sin ( \phi_j - \phi_i - \psi_\ast )
	\ne 0
\]
for $ j \ne i $.
We obtain
\[
	\begin{array}{rl}
	\cos \psi_\ast
	= & \zume
	\displaystyle{
	- \frac 1 { n_i }
	\lim_{ s \to +0 }
	\sum_{ j \ne i }
	\frac
	{ n_j \cos ( \phi_j - \phi_i - \psi_\ast ) v(s) }
	{ \vece ( \phi_j - \phi_i )^\bot \cdot \vecu (s) }
	}
	\\
	= & \zume
	\displaystyle{
	- \frac 1 { n_i }
	\lim_{ s \to +0 }
	\sum_{ j \ne i }
	\frac
	{ n_j \cos ( \phi_j - \phi_i - \psi_\ast ) v^\prime (s) }
	{ \vece ( \phi_j - \phi_i )^\bot \cdot \vecu^\prime (s) }
	}
	\\
	= & \zume
	\displaystyle{
	- \frac { \sin \psi_\ast } { n_i }
	\sum_{ j \ne i }
	\frac
	{ n_j \cos ( \phi_j - \phi_i - \psi_\ast) }
	{ \vece ( \phi_j - \phi_i )^\bot \cdot \vece ( \psi_\ast ) }
	}
	\\
	= & \zume
	\displaystyle{
	- \frac { \sin \psi_\ast } { n_i }
	\sum_{ j \ne i }
	\frac
	{ n_j \cos ( \phi_j - \phi_i - \psi_\ast ) }
	{ - \sin ( \phi_j - \phi_i - \psi_\ast ) }
	}
	\\
	= & \zume
	\displaystyle{
	-
	\frac { \sin \psi_\ast } { n_i }
	\sum_{ j \ne i }
	n_j \cot ( \psi_\ast + \phi_i - \phi_j )
	}
	\end{array}
\]
by L'Hospital's theorem.
This shows
\[
	\sum_{ j \in J }
	n_j \cot ( \psi_\ast + \phi_i - \phi_j )
	= 0
	,
\]
and therefore $ \psi_\ast = \theta_i - \phi_i $.
\qed
\end{proof}
\begin{lem}
On a neighborhood of $ s \to + 0 $,
we have
\[
	u^\prime (s) > 0 ,
	\quad
	v^\prime (s) > 0 ,
	\quad
	0 < \frac { v^\prime (s) } { u^\prime (s) } < \tan ( \phi_{ i+1 } - \phi_i )
	.
\]
\end{lem}
\begin{proof}
Since $ \displaystyle{ \lim_{ s \to + 0 } v(s) = 0 } $,
and since $ v(s) > 0 $,
$ v^\prime (s) \ne 0 $ for small $ s > 0 $,
we have
\[
	v^\prime (s) = \vece (0)^\bot \cdot \vecu^\prime (s) > 0
	.
\]
Applying a similar argument to $ R ( - \phi_{ i+1 } ) \vecx $,
we have
\[
	- \vece ( \phi_{ i+1 } - \phi_i )^\bot \cdot \vecu^\prime (s) > 0
	,
\]
{\em i.e.},
\[
	u^\prime (s) \sin ( \phi_{ i+1 } - \phi_i )
	-
	v^\prime (s) \cos ( \phi_{ i+1 } - \phi_i )
	> 0
	.
\]
Because of $ 0 < \phi_{ i+1 } - \phi_i \leqq \frac \pi 2 $,
we have $ \sin ( \phi_{ i+1 } - \phi_i ) > 0 $,
$ \cos ( \phi_{ i+1 } - \phi_i ) \geqq 0 $.
Therefore,
we obtain
\[
	u^\prime (s)
	>
	\frac { v^\prime (s) \cos ( \phi_{ i+1 } - \phi_i ) }
	{ \sin ( \phi_{ i+1 } - \phi_i ) }
	\geqq 0
	,
	\quad
	0 < \frac { v^\prime (s) } { u^\prime (s) }
	< \tan ( \phi_{ i+1 } - \phi_i )
\]
\qed
\end{proof}
\begin{cor}
It holds that
\[
	0 \leqq
	\liminf_{ s \to + 0 } \frac { v^\prime (s) } { u^\prime (s) }
	\leqq
	\liminf_{ s \to + 0 } \frac { v (s) } { u (s) }
	\leqq
	\limsup_{ s \to + 0 } \frac { v (s) } { u (s) }
	\leqq
	\limsup_{ s \to + 0 } \frac { v^\prime (s) } { u^\prime (s) }
	\leqq
	\tan ( \phi_{ i+1 } - \phi_i )
	.
\]
\label{boundsofliminfsup}
\end{cor}
\begin{proof}
It is by virtue of previous lemma and L'Hospital's theorem for limit superior and limit inferior.
\qed
\end{proof}
\begin{lem}
There exists the limit of $ \displaystyle{ \frac { \vecu (s) } { \| \vecu (s) \| } } $ as $ s \to + 0 $.
\end{lem}
\begin{proof}
Put
\[
	w(s) = \frac { v(s) } { u(s) },
	\quad
	z(s) = \frac { v^\prime (s) } { u^\prime (s) } ,
	\quad
	\liminf_{ s \to b - 0 } w(s) = \underline L ,
	\quad
	\limsup_{ s \to b - 0 } w(s) = \bar L ,
\]
and
\[
	L = \tan ( \theta_i - \phi_i ) .
\]
\par
Assume $ \underline L \ne \bar L $ and $ \underline L < L $.
And,
then,
taking into consideration of the shape of the geretating curve,
there exist sequences $ \{ s_j \} $ and $ \{ \tilde s_j \} $ such that
\[
	s_j > \tilde s_j > s_{ j+1 } > \tilde s_{ j+1 } ,
	\quad
	\lim_{ j \to \infty } s_j = + 0 ,
	\quad
	\lim_{ j \to \infty } \tilde s_j = + 0 ,
\]
\[
	\lim_{ j \to \infty } w( s_j ) = \underline L ,
	\quad
	\lim_{ j \to \infty } w ( \tilde s_j ) = \bar L ,
\]
\[
	\mbox{the generating curve is tangent to the line $ v = L_j u $ at $ s = s_j $,
	and } \lim_{ j \to \infty } L_j = \underline L .
\]
The last property implies
\[
	z ( s_j ) = L_j \to \underline L \quad \mbox{as} \quad j \to \infty .
\]
Put
\[
	B_\epsilon = \{ ( w,z ) \in \mathbb{R}^2 \, | \, ( w - \underline L )^2 + ( z - \underline L )^2 < \epsilon^2 \} .
\]
If $ \epsilon > 0 $ is sufficiently small,
then we may assume that
\[
	( w( s_j ) , z( s_j ) ) \in B_\epsilon ,
	\quad
	( w( \tilde s_j ) , z( \tilde s_j ) ) \in B_\epsilon^c .
\]
Hence,
there exists $ \{ \hat s_j \} $ such that
\[
	s_j > \hat s_j > \tilde s_j ,
	\quad
	( w(s) , z(s) ) \in \bar B_\epsilon \quad \mbox{for} \quad s \in ( \hat s_j , s_j ] ,
	\quad
	( w( \hat s_j ) , z ( \hat s_j ) ) \in \partial B_\epsilon .
\]
Now we consider the behavior of $ ( w(s) , z(s) ) $ on the interval $ I_j = [ \hat s_j , s_j ] $.
Then
\[
	\begin{array}{rl}
	\displaystyle{
	\frac 12 \frac d { ds } ( w(s) - L )^2
	}
	= & \!\!\!
	\displaystyle{
	( w(s) -  L ) w^\prime (s)
	}
	\\
	= & \!\!\!
	\displaystyle{
	( w(s) - L ) 
	\frac { v^\prime (s) u(s) - v(s) u^\prime (s) } { u(s)^2 }
	}
	\\
	= & \!\!\!
	\displaystyle{
	\frac { u^\prime (s) } { u(s) }
	( w(s) - L ) 
	( z(s) - w(s) ) .
	}
	\end{array}
\]
When $ s \in I_j $,
\[
	| w(s) - L | \leqq C
	,
\]
\[
	| z(s) - w(s) |
	=
	| z(s) - \underline L - ( w(s) - \underline L ) |
	\leqq
	2 \epsilon .
\]
Therefore,
\[
	\left|
	\frac { u^\prime (s) } { u(s) }
	-
	\frac { v^\prime (s) } { v(s) }
	\right|
	=
	| w(s) - z(s) | \left| \frac { u^\prime (s) } { v(s) } \right|
	\leqq
	2 \epsilon \left| \frac { u^\prime (s) } { v(s) } \right| ,
\]
which implies
\[
	\frac { u^\prime (s) } { u(s) }
	=
	\frac { v^\prime (s) + O ( \epsilon ) u^\prime (s) } { v(s) }
	.
\]
Consequently,
\[
	\left|
	\frac 12 \frac d { ds } ( w(s) - L )^2
	\right|
	=
	\left|
	\frac { v^\prime (s) + O ( \epsilon ) y^\prime (s) } { v(s) }
	\right|
	O ( \epsilon )
	=
	\frac { O( \epsilon ) } { v(s) } .
\]
Here we use $ | u^\prime (s) | \leqq 1 $,
$ | v^\prime (s) | \leqq 1 $.
On the other hand
\[
	\begin{array}{rl}
	\displaystyle{
	\frac 12 \frac d { ds } ( z(s) - L )^2
	}
	= & \!\!\!
	\displaystyle{
	( z(s) - L ) z^\prime (s)
	}
	\\
	= & \!\!\!
	\displaystyle{
	( z(s) - L )
	\frac { v^{\prime\prime} (s) u^\prime (s) - v^\prime (s) u^{ \prime \prime } (s) } { ( u^\prime (s) )^2 }
	}
	\\
	= & \!\!\!
	\displaystyle{
	- \frac { z(s) - L } { ( u^\prime (s) )^2 }
	\vecu^{ \prime \prime } (s)^\bot \cdot \vecu^\prime (s)
	}
	\\
	= & \!\!\!
	\displaystyle{
	- \frac { z(s) - L } { ( u^\prime (s) )^2 }
	\left\{
	- \sum_{ j \in J }
	\frac { n_j \vece ( \phi_j - \phi_i ) \cdot \vecu^\prime (s) }
	{ \vece ( \phi_j - \phi_i )^\bot \cdot \vecu (s) }
	+
	( n - 1 ) H(s)
	\right\}
	}
	.
	\end{array}
\]
Define $ \hat \theta $ and $ \check \theta $ by $ w(s) = \tan ( \hat \theta (s) - \phi_i ) $,
and $ z(s) = \tan ( \check \theta (s) - \phi_i ) $.
Then
\[
	\vecu (s) = \frac { u(s) } { \cos ( \hat \theta (s) - \phi_i ) } \vece ( \hat \theta (s) - \phi_i ) ,
	\quad
	\vecu^\prime (s) = \frac { u^\prime (s) } { \cos ( \check \theta (s) - \phi_i ) } \vece ( \hat \theta (s) - \phi_i ) ,
\]
\[
	\frac { \vece ( \phi_j - \phi_i ) \cdot \vecu^\prime (s) }
	{ \vece ( \phi_j - \phi_i )^\bot \cdot \vecu (s) }
	=
	\frac { u^\prime (s) w(s) \cos ( \check \theta (s) - \phi_i ) \cos ( \check \theta (s) - \phi_j ) }
	{ v(s) \cos ( \hat \theta (s) - \phi_i ) \sin ( \hat \theta (s) - \phi_j ) }
	.
\]
Define $ \underline \theta $ by $ \underline L = \tan ( \underline \theta - \phi_i ) $.
When $ ( w(s) , z(s) ) \in B_\epsilon $,
we have
\[
	\hat \theta (s) = \underline \theta + O( \epsilon ),
	\quad
	\check \theta (s) = \underline \theta + O( \epsilon ) .
\]
Hence,
\[
	\begin{array}{rl}
	\displaystyle{
	\sum_{ j \in J }
	\frac { n_j \vece ( \phi_j - \phi_i ) \cdot \vecu^\prime (s) }
	{ \vece ( \phi_j - \phi_i )^\bot \cdot \vecu (s) }
	}
	= & \zume
	\displaystyle{
	\frac { u^\prime (s) \cos ( \check \theta (s) - \phi_i ) }
	{ u(s) \cos ( \hat \theta (s) - \phi_i ) }
	\sum_{ j \in J }
	n_j \left( \cot ( \underline \theta - \phi_j )
	+ O( \epsilon )
	\right)
	}
	\\
	= & \zume
	\displaystyle{
	\frac { u^\prime (s) \cos ( \check \theta (s) - \phi_i ) }
	{ u(s) \cos ( \hat \theta (s) - \phi_i ) }
	A( \underline \theta )
	+ O ( \epsilon )
	}
	.
	\end{array}
\]
Our assumption $ \underline L < L $ implies $ \underline \theta < \theta_i $,
and therefore $ A( \underline \theta ) > A( \theta_i ) = 0 $.
There exists $ \lambda \in [ 0 , 1 ) $ such that
\[
	0 < w(s) < \frac { ( 1 + \lambda ) L } 2 ,
	\quad
	0 < z(s) < \frac { ( 1 + \lambda ) L } 2 ,
	\quad
	0 < u^\prime (s) \leqq 0 ,
	\quad
	v(s) > 0
	,
\]
\[
	\cos ( \hat \theta (s) - \phi_i ) )
	=
	\cos ( \underline \theta - \phi_i ) + O ( \epsilon ) ,
	\quad
	\cos ( \check \theta (s) - \phi_i ) )
	=
	\cos ( \underline \theta - \phi_i ) + O ( \epsilon )
\]
hold on $ I_j $ for large $ j $.
Hence,
there exists $ \delta > 0 $ independent of $ \epsilon $ such that
\[
	\frac { z(s) - L } { ( u^\prime (s) )^2 }
	\sum_{ j \in J }
	\frac { n_j \vece ( \phi_j - \phi_i ) \cdot \vecu^\prime (s) }
	{ \vece ( \phi_j - \phi_i )^\bot \cdot \vecu (s) }
	\leqq
	- \frac \delta { v(s) } .
\]
On the interval $ I_j $,
\[
	\frac 1 { ( x^\prime (s) )^2 }
	=
	\frac { \underline L^2 ( 1 + o(1) ) } { ( y^\prime (s) )^2 }
	=
	\frac { \underline L^2 ( 1 + o(1) ) } { 1 - ( x^\prime (s) )^2 } .
\]
If there exists a sequence $ \displaystyle{ \{ \bar s_k \} \subset \bigcup_j I_j } $ such that
\[
	\lim_{ k \to \infty } \bar s_k = + 0 ,
	\quad
	\lim_{ k \to \infty } u^\prime ( \bar s_k ) = 0,
\]
then as $ k \to \infty $
\[
	\infty \leftarrow \frac 1 { ( u^\prime ( \bar s_k ) )^2 }
	=
	\frac { \underline L^2 ( 1 + o(1) ) } { 1 - ( u^\prime ( \bar s_k ) )^2 }
	\to \underline L^2 < L^2 .
\]
This is a contradiction, 
 therefore we may assume
\[
	\inf \left\{ ( u^\prime (s) )^2 \, \left| \,
	s \in \bigcup_j I_j \right. \right\} > 0 .
\]
Hence,
\[
	\left| \frac { ( n - 1 ) H(s) ( z(s) - L ) } { ( u^\prime (s) )^2 }
	\right|
	\leqq C .
\]
Consequently,
\[
	\frac 12 \frac d { ds }
	\left\{
	( z(s) - L )^2
	+
	( w(s) - L )^2
	\right\}
	\leqq
	- \frac 1 { v(s) }
	\left( \delta + O( \epsilon ) \right)
	+ C
\]
on $ I_j $.
If $ j $ is sufficiently large,
then $ v(s) > 0 $ is sufficiently small.
Taking $ \epsilon $ small,
we have
\[
	\frac 12 \frac d { ds }
	\left\{
	( z(s) - L )^2
	+
	( w(s) - L )^2
	\right\}
	\leqq
	- \frac \delta { 2 v(s) } < 0
\]
on $ I_j $ for large $ j $.
Hence,
\[
	\begin{array}{rl}
	\displaystyle{
	( w( \hat s_j ) - L )^2
	+
	( z( \hat s_j ) - L )^2
	}
	\geqq & \!\!\!
	\displaystyle{
	( z( s_j ) - L )^2
	+
	( w( s_j ) - L )^2
	}
	\\
	= & \!\!\!
	\displaystyle{
	2 ( L_j - L )^2 .
	}
	\end{array}
\]
Taking a suitable subsequence,
we have $ ( w( \hat s_j ) , z( \hat s_j ) ) \to ( \hat w , \hat z ) $,
where
\[
	( \hat w , \hat z )
	\in
	\partial B_\epsilon
	\cap
	\left\{
	( w , z ) \in \mathbb{R}^2 \, \left| \,
	( w - L )^2
	+
	( z - L )^2
	\geqq
	2 ( \underline L - L )^2 \right. \right\} .
\]
This shows that
\[
	\hat w < \underline L
	\quad \mbox{or} \quad
	\hat z < \underline L .
\]
This leads to a contradiction.
Indeed,
if $ \hat w < \underline L $,
then
\[
	\liminf_{ s \to + 0 } w(s) = \underline L
	> \hat w
	= \lim_{ j \to \infty } w( \hat s_j )
	\geqq
	\liminf_{ s \to + 0 } w(s) .
\]
If $ \hat z < \underline L $,
then
\[
	\liminf_{ s \to + 0 } z(s) = \underline L
	> \hat z
	= \lim_{ j \to \infty } z( \hat s_j )
	\geqq
	\liminf_{ s \to + 0 } z(s) .
\]
\par
Now we go back to the 5th line of the proof,
and $ \underline L = \bar L $ or $ L \leqq \underline L $ has been proved.
\par
Similarly we have $ \underline L = \bar L $ or $ \bar L \leqq L $.
\par
Combining these,
we finally get $ \underline L = \bar L $.
\qed
\end{proof}
\par
Put
\[
	A( \alpha , \beta )
	=
	\sum_{ j \in J } \frac { n_j \cos ( \alpha - \phi_j ) } { \sin ( \beta - \phi_j ) }
	.
\]
Here we assume $ \beta \ne \phi_j $ for all $ j \in J $.
It follows that
\[
	\frac { \partial A } { \partial \alpha }
	=
	- \sum_{ j \in J }
	\frac { n_j \sin ( \alpha - \phi_j ) } { \sin ( \beta - \phi_j ) }
	.
\]
When $ \phi_i < \alpha < \phi_{ i+1 } $ and $ \phi_i < \beta < \phi_{ i+1 } $,
it happens that
\[
	\mbox{\rm sgn} \sin ( \alpha - \phi_j )
	=
	\mbox{\rm sgn} \sin ( \beta - \phi_j )
	.
\]
Therefore,
we have
\[
	\frac { \partial A } { \partial \alpha } < 0
	.
\]
\begin{lem}
There exists the limit of $ \vecu^\prime (s) $ as $ s \to + 0 $ and
\[
	\lim_{ s \to + 0 } \vecu^\prime (s)
	=
	\lim_{ s \to + 0 } \frac { \vecu (s) } { \| \vecu (s) \| }
	=
	\vece ( \theta_i - \phi_i )
	.
\]
\end{lem}
\begin{proof}
It is enough to show
\[
	\lim_{ s \to + 0 } \frac { v^\prime (s) } { u^\prime (s) }
	=
	\lim_{ s \to + 0 } \frac { v (s) } { u (s) }
	=
	\tan ( \theta_i - \phi_i )
	.
\]
If $ \displaystyle{ \frac { v^\prime (s) } { u^\prime (s) } } $ is monotone near $ s = + 0 $,
then there exists $ \displaystyle{ \lim_{ s \to + 0 } \frac { v^\prime (s) } { u^\prime (s) } } $.
\par
Otherwise, we put
\[
	\liminf_{ s \to + 0 } \frac { v^\prime (s) } { u^\prime (s) }
	=
	\tan ( \underline \theta^\prime - \phi_i ) ,
	\quad
	\limsup_{ s \to + 0 } \frac { v^\prime (s) } { u^\prime (s) }
	=
	\tan ( \bar \theta^\prime - \phi_i )
	.
\]
There exists a sequence $ \{ s_k \} $ such that $ \displaystyle{ v^\prime (s) } { u^\prime (s) } $ takes a minimum value at $ s = s_k $ and
\[
	s_k \to 0 ,
	\quad
	\frac { v^\prime ( s_k ) } { u^\prime ( s_k ) } \to \tan ( \underline \theta^\prime - \phi_i ) ,
	\quad
	\left( \frac { v^\prime } { u^\prime } \right)^\prime ( s_k ) = 0
\]
as $ k \to \infty $.
From the third relation it follows that $ \vecu^{ \prime \prime }( s_k )^\bot \cdot \vecu^\prime ( s_k ) = 0 $.
By using the equation we have
\[
	\| \vecu ( s_k ) \| ( n - 1 ) H( s_k )
	=
	\sum_{ j \in J }
	\frac { n_j \vece ( \phi_j - \phi_i ) \cdot \vecu^\prime ( u_k ) }
	{ \vece ( \phi_j - \phi_i )^\bot \cdot \frac { \vecu ( s_k ) } { \| \vecu ( s_k ) \| } }
	\to A ( \underline \theta^\prime , \bar \theta )
\]
as $ k \to \infty $.
Because of the boundedness of $ H $,
it is clear that
\[
	\| \vecu ( s_k ) \| ( n - 1 ) H( s_k ) \to 0
	.
\]
Hence,
we have $ A( \underline \theta^\prime , \bar \theta ) = 0 $.
Using a sequence of $ s $ where $ \displaystyle{ v^\prime (s) } { u^\prime (s) } $ takes a maximum value,
we get $ A( \bar \theta^\prime , \bar \theta ) = 0  $.
Combining these,
we know $ A ( \underline \theta^\prime , \bar \theta ) = A( \bar \theta^\prime , \bar \theta ) $.
Since $ \displaystyle{ \frac { \partial A } { \partial \alpha } < 0 } $,
we obtain $ \underline \theta^\prime = \bar \theta^\prime $.
\par
Consequently,
in any of these cases,
$ \displaystyle{ \lim_{ s \to + 0 } \frac { v^\prime (s) } { u^\prime (s) } } $ exists.
From Lemma \ref{iflimitexists} it follows that the limit value is $ \tan ( \theta_i - \phi_i ) $.
Finally we know $ \displaystyle{ \lim_{ s \to + 0 } \frac { v(s) } { u(s) } = \tan ( \theta_i - \phi_i ) } $ by Corollary \ref{boundsofliminfsup}.
\qed
\end{proof}
\par
Thus we have completed the proof of Proposition \ref{prop4.1}.
\par
Next we prove the converse of Proposition \ref{prop4.1},
{\em i.e.},
the solvability of \pref{equationx} and \pref{initialcondition-ii}.
We define the function $ q $ as before.
For case (ii),
\[
	\lim_{ v \to 0 } q(v) = \cot ( \theta_i - \phi_i )
	.
\]
Hence,
we introduce new unknown functions $ r $ and $ \rho $ by
\[
	r(v) = q(v) - \cot ( \theta_i - \phi_i ),
	\quad
	\rho (v) = \frac 1 v \int_0^v r ( \eta ) \, d \eta
	.
\]
Then our problem is equivalent to
\be
	\left\{
	\begin{array}{l}
	\displaystyle{
	v \frac { dr } { dv } (v)
	+ ( n - 2 ) r(v)
	=
	- \gamma ( n - 2 ) \rho (v)
	+
	\sum_{ k=1 }^5
	F_k ( r, \rho ) (v)
	}
	,
	\\
	r(0) = \rho (0) = 0
	,
	\end{array}
	\right.
	\label{equationrrho1}
\ee
where
\[
	\gamma = \sharp J - 1
	,
\]
\[
	F_1 ( r, \rho )
	= F_1 (r)
	=
	- ( n - 2 ) r^2
	\left\{ r + 2 \cot ( \theta_i - \phi_i ) \right\}
	\sin^2 ( \theta_i - \phi_i )
	,
\]
\[
	F_2 ( r, \rho )
	=
	- \gamma ( n - 2 ) r \rho \left\{ r + 2 \cot ( \theta_i - \phi_i ) \right\}
	\sin^2 ( \theta_i - \phi_i )
	,
\]
\[
	\begin{array}{l}
	F_3 ( r, \rho )
	=
	\displaystyle{
	-
	r \rho \left\{
	r^2 + 2 r \cot ( \theta_i - \phi_i ) + \mathrm{cosec}^2 ( \theta_i - \phi_i )
	\right\}
	}
	\\
	\qquad
	\displaystyle{
	\times
	\sum_{ j \in J }
	\frac
	{ n_j \sin^2 ( \theta_i - \phi_i ) \cos ( \phi_j - \phi_i ) \sin ( \phi_j - \phi_i ) }
	{ \displaystyle{
	\sin ( \theta_i - \phi_j )
	\left\{
	\sin ( \theta_i - \phi_j ) - \rho \sin ( \phi_j - \phi_i ) \sin ( \theta_i - \phi_i )
	\right\}
	} }
	}
	\end{array}
\]
\[
	\begin{array}{l}
	F_4 ( r, \rho )
	=
	\displaystyle{
	-
	\rho^2
	\left\{
	r^2 + 2 r \cot ( \theta_i - \phi_i ) + \mathrm{cosec}^2 ( \theta_i - \phi_i )
	\right\}
	}
	\\
	\qquad
	\displaystyle{
	\times
	\sum_{ j \in J }
	\frac
	{ n_j \sin^2 ( \theta_i - \phi_i ) \sin^2 ( \phi_j - \phi_i ) \cos ( \theta_i - \phi_j ) }
	{ \displaystyle{
	\sin^2 ( \theta_i - \phi_j )
	\left\{
	\sin ( \theta_i - \phi_j ) - \rho \sin ( \phi_j - \phi_i ) \sin ( \theta_i - \phi_i )
	\right\}
	} }
	}
	,
	\end{array}
\]
and
\[
	F_5 ( r, \rho ) (v)
	=
	F_5 (r) (v)
	=
	( n - 1 )
	\left[ \left\{ \cot ( \theta_i - \phi_i ) + r(v) \right\}^2 + 1 \right]^{ \frac 32 } 
	\tilde H(v) v
	.
\]
Our derivation of \pref{equationrrho1} is elementary but it needs lengthy calculations,
so we present it in the Appendix.
\par
Multiplying both sides of the first equation in \pref{equationrrho1} by $ v^{ n - 3 } $,
and integrating from $ 0 $ to $ v $,
we have
\[
	r(v)
	=
	- \frac { \gamma ( n - 2 ) } { v^{ n - 2 } }
	\int_0^v \rho ( \eta ) \eta^{ n - 3 } d \eta
	+
	\frac 1 { v^{ n - 2 } } \int_0^v \sum_{ k=1 }^5 \psi_k ( r, \rho ) ( \eta ) \, d \eta
	,
\]
where
\[
	\psi_k ( r, \rho ) ( \eta ) = F_k ( r, \rho ) ( \eta ) \eta^{ n - 3 }
	.
\]
\par
Since the function $ \rho $ defined by $ r $,
we can define the map $ \Psi $ by
\[
	\Psi (r) (v)
	=
	- \frac { \gamma ( n - 2 ) } { v^{ n - 2 } }
	\int_0^v \rho ( \eta ) \eta^{ n - 3 } d \eta
	+
	\frac 1 { v^{ n - 2 } } \int_0^v \sum_{ k=1 }^5 \psi_k ( r, \rho ) ( \eta ) \, d \eta
	.
\]
Taking $ M $ large,
and $ V $ small,
we can show this is a contraction map from $ X_{V,M} $ into itself for Types II--III.
This fact can be proved in the same way as \cite{KenNaga}.
Indeed,
the principal part of $ \Psi $ is the map
\be
	\bar \Psi \, : \, r \mapsto 
	- \frac { \gamma ( n -2 ) } { v^{ n - 2 } }
	\int_0^v \rho ( \eta ) \eta^{ n - 3 } d \eta .
	\label{principalpartofPsi}
\ee
Because
\[
	\begin{array}{rl}
	\displaystyle{
	\left|
	- \frac { \gamma ( n - 2 ) } { v^{ n-1 } }
	\int_0^v ( \rho_1 ( \eta ) - \rho_2 ( \eta ) ) \eta^{ n-3 } d \eta
	\right|
	}
	\leqq & \zume
	\displaystyle{
	\frac { \gamma ( n - 2 ) \| \rho_1 - \rho_2 \| } { v^{ n-1 } }
	\int_0^v \eta^{ n-2 } d \eta
	}
	\\
	\leqq & \zume
	\displaystyle{
	\frac { \gamma ( n - 2 ) } { 2 ( n - 1 ) } \| r_1 - r_2 \|
	}
	,
	\end{array}
\]
The map $ \bar \Psi $ is contractive when $ \gamma \leqq 2 $,
which holds for Types II--III.
Since $ \Psi $ is a small perturbation of $ \bar \Psi $,
it is also contractive for Types II--III.
Testing with  the linear functions $ r_i = c_i v $,
we find that the map $ \bar \Psi $ is expansive for Types IV--V.
This suggests that $ \Psi $ may not be contractive for these types,
and therefore we must deal with our problem more carefully.
\par
Since
\[
	\frac { d \rho } { dv }
	=
	- \frac 1 { v^2 } \int_0^v r \, d \eta + \frac rv
	=
	- \frac \rho v + \frac rv
	,
\]
we have
\[
	v \frac d { dv }
	\left(
	\begin{array}{c}
	r \\ \rho
	\end{array}
	\right)
	+
	\left(
	\begin{array}{cc}
	n - 2  & \gamma ( n - 2 ) \\
	-1 & 1
	\end{array}
	\right)
	\left(
	\begin{array}{c}
	r \\ \rho
	\end{array}
	\right)
	=
	\left(
	\begin{array}{c}
	\displaystyle{
	\sum_{ k=1 }^5 F_k ( r, \rho )
	}
	\\
	0
	\end{array}
	\right)
	.
\]
Eigenvalues of the matrix in the left-hand side are
\[
	\lambda_{\pm}
	=
	\frac { n - 1 \pm \sqrt{ n^2 - 2 ( 2 \gamma + 3 ) n + 8 \gamma + 9 } } 2
	.
\]
Since
\[
	( 2 \gamma + 3 )^2 -(8 \gamma + 9 ) = 4 \gamma ( \gamma + 1 ) > 0
	,
\]
We know $ \lambda_+ \ne \lambda_- $.
Therefore,
there exists a non-singular matrix $ P $ such that
\[
	P^{-1} \left( \begin{array}{cc}
	n - 2 & \gamma ( n - 2 ) \\
	-1 & 1
	\end{array} \right) P
	=
	\left( \begin{array}{cc}
	\lambda_+ & 0 \\
	0 & \lambda_-
	\end{array} \right)
	.
\]
Put
$ P = ( p_{ij} ) $,
$ P^{-1} = ( p^{ij} ) $.
These are matrices with constant entries.
Define $ \hat r $ and $ \hat \rho $ by
\[
	\left( \begin{array}{c}
	\hat r \\ \hat \rho
	\end{array} \right)
	=
	P^{-1} \left( \begin{array}{c}
	r \\ \rho
	\end{array} \right)
	.
\]
Then the equation can be rewritten as
\[
	v \frac d { dv }
	\left( \begin{array}{c}
	\hat r \\ \hat \rho
	\end{array} \right)
	+
	\left( \begin{array}{cc}
	\lambda_+ & 0 \\
	0 & \lambda_-
	\end{array} \right)
	\left( \begin{array}{c}
	\hat r \\ \hat \rho
	\end{array} \right)
	=
	\left( \begin{array}{c}
	\displaystyle{
	p^{11} \sum_{ k=1 }^5 \hat F_k ( \hat r , \hat \rho )
	}
	\\
	\displaystyle{
	p^{21} \sum_{ k=1 }^5 \hat F_k ( \hat r , \hat \rho )
	}
	\end{array}
	\right)
	,
	\quad
	\hat r (0) = \hat \rho (0) = 0
	,
\]
where
\[
	\hat F_k ( \hat r , \hat \rho )
	=
	F_k ( p_{11} \hat r + p_{12} \hat \rho , p_{21} \hat r + p_{22} \hat \rho )
	( = F_k ( r , \rho ) )
	.
\]
Hence,
we have
\[
	\hat r (v)
	=
	\frac { p^{11} } { v^{ \lambda_+ } } \int_0^v
	\sum_{ k=1 }^5 \hat F_k ( \hat r , \hat \rho ) ( \eta )
	\eta^{ \lambda_+ - 1 } d \eta ,
	\quad
	\hat \rho (v)
	=
	\frac { p^{21} } { v^{ \lambda_- } } \int_0^v
	\sum_{ k=1 }^5 \hat F_k ( \hat r , \hat \rho ) ( \eta )
	\eta^{ \lambda_- - 1 } d \eta
	.
\]
Considering the Banach spaces $ X_V $ defined earlier in this paper,  let $ X_V \times X_V $ be the Banach space with norm
\[
	\| ( \hat r , \hat \rho ) \|_{ X_V \times X_V }
	=
	\| \hat r \|_{ X_V } + \| \hat \rho \|_{ X_V }
	.
\]
Define the map $ \hat \Psi $ by
\[
	\hat \Psi ( \hat r , \hat \rho ) (v)
	=
	\left(
	\frac { p^{11} } { v^{ \lambda_+ } } \int_0^v
	\sum_{ k=1 }^5 \hat F_k ( \hat r , \hat \rho ) ( \eta )
	\eta^{ \lambda_+ - 1 } d \eta ,
	\frac { p^{21} } { v^{ \lambda_- } } \int_0^v
	\sum_{ k=1 }^5 \hat F_k ( \hat r , \hat \rho ) ( \eta )
	\eta^{ \lambda_- - 1 } d \eta
	\right)
	.
\]
We will show that if $ M $ is large,
and of $ V $ is small,
then $ \hat \Psi $ is a contraction map from $ X_{V,M} \times X_{V,M} $ into itself.
\par
When $ n^2 - 2 ( 2 \gamma + 3 ) n + 8 \gamma + 9 < 0 $,
\[
	\Re \lambda_\pm = \frac { n - 1 } 2 > 0 .
\]
If $ n^2 - 2 ( 2 \gamma + 3 ) n + 8 \gamma + 9 \geqq 0 $,
then
\[
	0 \leqq n^2 - 2 ( 2 \gamma + 3 ) n + 8 \gamma + 9
	= ( n - 1 )^2 - 4 ( \gamma + 1 ) ( n - 2 )
	< ( n - 1 )^2 ,
\]
and hence
\[
	\Re \lambda_\pm = \frac { n - 1 \pm \sqrt{ n^2 - 2 ( 2 \gamma + 3 ) n + 8 \gamma + 9 } } 2 > 0 .
\]
Therefore,
in any of the cases,
the integral $ \displaystyle{ \int_0^v \eta^{ \Re \lambda_\pm + p } d \eta } $ converges for $ p > -1 $,
and
\[
	\int_0^v \eta^{ \Re \lambda_\pm + p } d \eta = \frac { v^{ \Re \lambda_\pm + p + 1 } } { \Re \lambda_\pm + p + 1 } .
\]
\par
Since both $ P $ and $ P^{-1} $ are constant matrices,
it holds that
\[
	\| ( r , \rho ) \| \leqq C \| ( \hat r , \hat \rho ) \| ,
	\quad
	\| ( \hat r , \hat \rho ) \| \leqq C \| ( r , \rho ) \| .
\]
Let $ ( \hat r , \hat \rho ) \in X_{ V,M } \times X_{ V,M } $,
and then we have
\[
	\begin{array}{rl}
	| \hat F_1 ( \hat r , \hat \rho ) ( \eta )|
	= & \zume
	| F_1 ( r , \rho ) ( \eta ) |
	\leqq
	C \| r \|^2 \eta^2 \left( \| r \| \eta + 1 \right)
	\\
	\leqq &\zume
	C \| ( \hat r , \hat \rho ) \|^2 \eta^2 \left( \| ( \hat r , \hat \rho ) \| \eta + 1 \right)
	\leqq
	C M^2 \eta^2 \left( M \eta + 1 \right) .
	\end{array}
\]
Therefore,
we get
\[
	\begin{array}{rl}
	\displaystyle{
	\left|
	\frac 1 { v^{ \lambda_+ + 1 } } \int_0^v
	\hat F_1 ( \hat r , \hat \rho ) ( \eta ) \eta^{ \lambda_+ - 1 } d \eta
	\right|
	}
	\leqq & \zume
	\displaystyle{
	\frac C { v^{ \Re \lambda_+ + 1 } } \int_0^v
	M^2 \eta^2 \left( M \eta + 1 \right)
	\eta^{ \Re \lambda_+ - 1 } d \eta
	}
	\\
	\leqq & \zume
	\displaystyle{
	\frac C { v^{ \Re \lambda_+ + 1 } }
	\left( M^3 v^{ \Re \lambda_+ 3 } + M^2 v^{ \Re \lambda_+ 2 } \right)
	}
	\\
	\leqq & \zume
	\displaystyle{
	C \left( M^3 V^2 + M^2 V \right)
	}
	.
	\end{array}
\]
We can estimate $ \displaystyle{ \frac 1 { v^{ \lambda_+ + 1 } } \int_0^v \hat F_k ( \hat r , \hat \rho ) ( \eta ) \eta^{ \Re \lambda_+ - 1 } d \eta } $ for $ k = 2 $,
$ 3 $,
$ 4 $,
$ 5 $ in a similar manner,
and we  obtain
\[
	\left|
	\frac { p^{11} } { v^{ \Re \lambda_+ + 1 } }
	\int_0^v \sum_{ k=1 }^5 \hat F_k ( \hat r , \hat \rho ) ( \eta ) \eta^{ \lambda_+ - 1 } d \eta
	\right|
	\leqq
	C \left( M^3 V^2 + M^2 V + M^4 V^3 + M^3 V^3 + 1 \right) .
\]
We can also derive
\[
	\left|
	\frac { p^{21} } { v^{ \Re \lambda_- + 1 } }
	\int_0^v \sum_{ k=1 }^5 \hat F_k ( \hat r , \hat \rho ) ( \eta ) \eta^{ \lambda_- - 1 } d \eta
	\right|
	\leqq
	C \left( M^3 V^2 + M^2 V + M^4 V^3 + M^3 V^3 + 1 \right)
\]
in the same manner.
From these it follows that
\[
	\| \hat \Psi ( \hat r , \hat \rho ) \|
	\leqq
	C \left( M^3 V^2 + M^2 V + M^4 V^3 + M^3 V^3 + 1 \right) .
\]
Consequently,
$ \hat \Psi $ is a map from $ X_{V,M} \times X_{V,M} $ into itself provided $ M $ is large and $ V $ is small.
\par
Using
\[
	\| ( r_1 , \rho_1 ) - ( r_2 , \rho_2 ) \|
	\leqq
	C \| ( \hat r_1 , \hat \rho_1 ) - ( \hat r_2 , \hat \rho_2 ) \| ,
\]
\[
	\| ( \hat r_1 , \hat \rho_1 ) - ( \hat r_2 , \hat \rho_2 ) \|
	\leqq
	C \| ( r_1 , \rho_1 ) - ( r_2 , \rho_2 ) \|
	,
\]
we can get
\[
	\begin{array}{l}
	\| \hat \Psi ( \hat r_1 , \hat \rho_1 ) - \hat \Psi ( \hat r_2 , \rho_2 ) \|
	\\
	\quad
	\leqq
	C \left( M^2 V^2 + MV + M^4 V^4 + M^2 V^3 + V \right)
	\| ( \hat r_1 , \hat \rho_1 ) - ( \hat r_2 , \hat \rho_2 ) \|
	.
	\end{array}
\]
Indeed,
from
\[
	\begin{array}{l}
	\left| \hat F_1 ( \hat r_1 , \hat \rho_1 ) - \hat F_1 ( \hat r_2 , \hat \rho_2 ) \right|
	=
	\left| F_1 ( r_1 ) - F_1 ( r_2 ) \right|
	\\
	\quad
	=
	\displaystyle{
	\left|
	- ( n - 2 ) ( r_1 - r_2 )
	\left\{ r_1^2 + r_1 r_2 + r_2
	+ 2 ( r_1 + r_2 ) \cot ( \theta_i - \phi_i )
	\right\}
	\sin^2 ( \theta_i - \phi_i )
	\right|
	}
	\\
	\quad
	\leqq
	\displaystyle{
	C \| r_1 - r_2 \| \eta \left( M^2 \eta^2 + M \eta \right)
	}
	\\
	\quad
	\leqq
	\displaystyle{
	C \| \hat r_1 - \hat r_2 \| \eta \left( M^2 \eta^2 + M \eta \right)
	}
	\end{array}
\]
it follows that
\[
	\begin{array}{l}
	\displaystyle{
	\left|
	\frac { p^{11} } { v^{ \lambda_+ +1 } }
	\int_0^v
	\left(
	\hat F_1 ( r_1 , \rho_1 ) ( \eta ) - \hat F_1 ( r_2 , \rho_2 ) ( \eta )
	\right)
	\eta^{ \lambda_+ -1 } d \eta
	\right|
	}
	\\
	\quad
	\leqq
	\displaystyle{
	\frac { C \| \hat r_1 - \hat r_2 \| } { v^{ \Re \lambda_+ +1 } }
	\int_0^v \left( M^2 \eta^3 + M \eta^2 \right) \eta^{ \Re \lambda_+ -1 } d \eta
	}
	\\
	\quad
	\leqq
	\displaystyle{
	\frac { C \| \hat r_1 - \hat r_2 \| } { v^{ \Re \lambda_+ +1 } }
	\left( M^2 v^{ \Re \lambda_+ +3 } + M v^{ \Re \lambda_+ +2 } \right)
	}
	\\
	\quad
	\leqq
	\displaystyle{
	C \left( M^2 V^2 + MV \right) \| \hat r_1 - \hat r_2 \|
	}
	.
	\end{array}
\]
Similarly we have
\[
	\begin{array}{l}
	\displaystyle{
	\left|
	\frac { p^{11} } { v^{ \lambda_+ +1 } }
	\int_0^v
	\left(
	\hat F_2 ( r_1 , \rho_1 ) ( \eta ) - \hat F_2 ( r_2 , \rho_2 ) ( \eta )
	\right)
	\eta^{ \lambda_+ -1 } d \eta
	\right|
	}
	\\
	\quad
	\leqq
	\displaystyle{
	C \left( M^2 V^2 + MV \right) \| \hat r_1 - \hat r_2 \|
	}
	,
	\end{array}
\]
\[
	\begin{array}{l}
	\displaystyle{
	\left|
	\frac { p^{11} } { v^{ \lambda_+ +1 } }
	\int_0^v
	\left(
	\hat F_k ( r_1 , \rho_1 ) ( \eta ) - \hat F_k ( r_2 , \rho_2 ) ( \eta )
	\right)
	\eta^{ \lambda_+ -1 } d \eta
	\right|
	}
	\\
	\quad
	\leqq
	\displaystyle{
	C \left( M^4 V^4 + MV \right)
	\left( \| \hat r_1 - \hat r_2 \| + \| \hat \rho_1 - \hat \rho_2 \| \right)
	\quad
	\mbox{for} \quad  k = 3 , \, 4
	}
	,
	\end{array}
\]
\[
	\begin{array}{l}
	\displaystyle{
	\left|
	\frac { p^{11} } { v^{ \lambda_+ +1 } }
	\int_0^v
	\left(
	\hat F_5 ( r_1 , \rho_1 ) ( \eta ) - \hat F_5 ( r_2 , \rho_2 ) ( \eta )
	\right)
	\eta^{ \lambda_+ -1 } d \eta
	\right|
	}
	\\
	\quad
	\leqq
	\displaystyle{
	C \left( M^2 V^3 + V \right) \| \hat r_1 - \hat r_2 \|
	}
	.
	\end{array}
\]
Therefore,
it holds that
\[
	\begin{array}{l}
	\displaystyle{
	\left\|
	\frac { p^{11} } { v^{ \lambda_+ } }
	\int_0^v \sum_{ k=1 }^5
	\left( \hat F_k ( \hat r_1 , \hat \rho_1 ) ( \eta ) - \hat F_k ( \hat r_2 , \hat \rho_2 ) ( \eta ) \right)
	\eta^{ \lambda_+ - 1 } d \eta
	\right\|_{ X_V }
	}
	\\
	\quad
	\leqq
	\displaystyle{
	C \left( M^2 V^2 + MV + M^4 V^4 + M^2 V^3 + V \right)
	\| ( \hat r_1 , \hat \rho_1 ) - ( \hat r_2 , \hat \rho_2 ) \|
	} .
	\end{array}
\]
We can also derive 
\[
	\begin{array}{l}
	\displaystyle{
	\left\|
	\frac { p^{21} } { v^{ \lambda_- } }
	\int_0^v \sum_{ k=1 }^5
	\left( \hat F_k ( \hat r_1 , \hat \rho_1 ) ( \eta ) - \hat F_k ( \hat r_2 , \hat \rho_2 ) ( \eta ) \right)
	\eta^{ \lambda_- - 1 } d \eta
	\right\|_{ X_V }
	}
	\\
	\quad
	\leqq
	\displaystyle{
	C \left( M^2 V^2 + MV + M^4 V^4 + M^2 V^3 + V \right)
	\| ( \hat r_1 , \hat \rho_1 ) - ( \hat r_2 , \hat \rho_2 ) \|
	}
	\end{array}
\]
in the same manner.
\par
Consequently,
the map $ \hat \Psi $ is contraction if $ V $ is sufficiently small.
The unique fixed point is a local solution to \pref{equationrrho1}.
\par
Because eigenvalues $ \lambda_\pm $ and matrix $ P $ are not necessarily real,
our solution might not be real-valued.
Therefore,
we must conform that our $ r $ and $ \rho $ are real-valued.
Putting $ r_{{}_I} = \Im r $ and $ \rho_{{}_I} = \Im \rho $,
we want to show $ r_{{}_I} = \rho_{{}_I} \equiv 0 $.
It is easy to see that $ r_{{}_I} $ satisfies
\[
	v \frac d { dv } r_{{}_I }
	+ ( n - 2 ) r_{{}_I }
	+ \gamma ( n - 2 ) \rho_{{}_I }
	=
	\sum_{ k=1 }^5 \Im F_k ( r , \rho ) .
\]
Multiplying both sides by $ 2 r_{{}_I } $,
we have
\[
	\frac d { dv } \left[ v \left\{ r_{{}_I }^2 + \gamma ( n - 2 ) \rho_{{}_I}^2 \right\} \right]
	+ ( 2n - 5 ) r_{{}_I}^2 + \gamma ( n - 2 ) \rho_{{}_I}^2
	=
	2 r_{{}_I } \sum_{ k=1 }^5 \Im \hat F_k ( \hat r , \hat \rho ) .
\]
We shall use the following lemma first and give it's proof  inmediately after.
\begin{lem}
Assume that $ V $ is sufficiently small.
There exists a positive constant $ C $ depending on $ M $ and $ V $ such that
\[
	| \Im \hat F_k ( \hat r , \hat \rho ) |
	\leqq
	C v \left( | r_{{}_I } | + | \rho_{{}_I } | \right)
\]
\label{ImFk}
\end{lem}
\par
Using the lemma,
$ n - 2 > 0 $,
and $ \gamma > 0 $,
we have
\[
	v \left\{ r_{{}_I }^2 + \gamma ( n - 2 ) \rho_{{}_I}^2 \right\}
	\leqq
	C \int_0^v \eta \left\{ r_{{}_I }^2 + \gamma ( n - 2 ) \rho_{{}_I}^2 \right\} d \eta .
\]
From Gronwall's lemma it follows that
\[
	r_{{}_I } \equiv 0 ,
	\quad
	\rho_{{}_I} \equiv 0 .
\]
\par\noindent
{\em Proof of Lemme \ref{ImFk} for $ \Im \hat F_1 $,
$ \cdots $,
$ \Im \hat F_4 $}.
Put $ \Re r = r_{{}_R} $,
$ \Re \rho = \rho_{{}_R} $.
Since $ r_{{}_R } $,
$ r_{{}_I } $,
$ \rho_{{}_R} $,
$ \rho_{{}_I} \in X_{V,M} $,
these moduli are dominated by $ C v $.
Therefore,
we have
\[
	\begin{array}{rl}
	| \Im \hat F_1 ( \hat r , \hat \rho ) |
	= & \zume
	| \Im F_1 (r) |
	\leqq
	C \left\{
	\left| \Im \left( r^3 \right) \right|
	+
	\left| \Im \left( r^2 \right) \right|
	\right\}
	\\
	\leqq & \zume
	C \left(
	\left| r_{{}_R}^2 r_{{}_I} \right|
	+
	\left| r_{{}_I}^3 \right|
	+
	| r_{{}_R} r_{{}_I} |
	\right)
	\\
	\leqq & \zume
	C \left( v^2 + v \right) | r_{{}_I} |
	\\
	\leqq & \zume
	C v | r_{{}_I} |
	,
	\end{array}
\]
\[
	\begin{array}{rl}
	| \Im \hat F_2 ( \hat r , \hat \rho ) |
	= & \zume
	| \Im F_2 ( r , \rho ) |
	\leqq
	C \left\{
	\left| \Im \left( r^2 \rho \right) \right|
	+
	\left| \Im \left( r \rho \right) \right|
	\right\}
	\\
	\leqq & \zume
	C \left(
	r^2 | r_{{}_I} |
	+
	| r_{{}_R } \rho_{{}_R} | | r_{{}_I} |
	+
	| \rho_{{}_R} | | r_{{}_I} |
	+
	| r_{{}_R} | | \rho_{{}_I} |
	\right)
	\\
	\leqq & \zume
	C \left( v^2 + v \right) \left( | r_{{}_I} | + | \rho_{{}_I } | \right)
	\\
	\leqq & \zume
	C v \left( | r_{{}_I} | + | \rho_{{}_I } | \right)
	.
	\end{array}
\]
$ \Im \hat F_3 $ is estimated as follows:
\[
	\begin{array}{rl}
	| \Im \hat F_3 ( \hat r , \hat \rho ) |
	= & \zume
	| \Im F_3 ( r , \rho ) |
	\\
	\leqq & \zume
	\displaystyle{
	\left| \Im \left[
	r \rho \left\{
	r^2 + 2 r \cot ( \theta_i - \phi_i ) + \mathrm{cosec}^2 ( \theta_i - \phi_i )
	\right\}
	\right]
	\right|
	}
	\\
	& \qquad
	\displaystyle{
	\times
	\left|
	\Re \left[
	\sum_{ j \in J }
	\frac
	{ n_j \sin^2 ( \theta_i - \phi_i ) \cos ( \phi_j - \phi_i ) \sin ( \phi_j - \phi_i ) }
	{ \displaystyle{
	\sin ( \theta_i - \phi_j )
	\left\{
	\sin ( \theta_i - \phi_j ) - \rho \sin ( \phi_j - \phi_i ) \sin ( \theta_i - \phi_i )
	\right\}
	} }
	\right]
	\right|
	}
	\\
	& \quad
	\displaystyle{
	+ \,
	\left| \Re \left[
	r \rho \left\{
	r^2 + 2 r \cot ( \theta_i - \phi_i ) + \mathrm{cosec}^2 ( \theta_i - \phi_i )
	\right\}
	\right]
	\right|
	}
	\\
	& \qquad
	\displaystyle{
	\times
	\left|
	\Im \left[
	\sum_{ j \in J }
	\frac
	{ n_j \sin^2 ( \theta_i - \phi_i ) \cos ( \phi_j - \phi_i ) \sin ( \phi_j - \phi_i ) }
	{ \displaystyle{
	\sin ( \theta_i - \phi_j )
	\left\{
	\sin ( \theta_i - \phi_j ) - \rho \sin ( \phi_j - \phi_i ) \sin ( \theta_i - \phi_i )
	\right\}
	} }
	\right]
	\right|
	}
	.
	\end{array}
\]
We can estimate each term as follows:
\[
	\begin{array}{l}
	\displaystyle{
	\left|
	\Im
	\left[
	r \rho \left\{
	r^2 + 2 r \cot ( \theta_i - \phi_i ) + \mathrm{cosec}^2 ( \theta_i - \phi_i )
	\right\}
	\right]
	\right|
	}
	\\
	\quad
	\leqq
	\displaystyle{
	C \left(
	\left| \Im \left( r^3 \rho \right) \right|
	+
	\left| \Im \left( r^2 \rho \right) \right|
	+
	\left| \Im ( r \rho ) \right|
	\right)
	}
	\\
	\quad
	\leqq
	\displaystyle{
	C v \left( | r_{{}_I } | + | \rho_{{}_I } | \right)
	}
	,
	\end{array}
\]
\[
	\left|
	\Re \left[
	\sum_{ j \in J }
	\frac
	{ n_j \sin^2 ( \theta_i - \phi_i ) \cos ( \phi_j - \phi_i ) \sin ( \phi_j - \phi_i ) }
	{ \displaystyle{
	\sin ( \theta_i - \phi_j )
	\left\{
	\sin ( \theta_i - \phi_j ) - \rho \sin ( \phi_j - \phi_i ) \sin ( \theta_i - \phi_i )
	\right\}
	} }
	\right]
	\right|
	\leqq C ,
\]
\[
	\left|
	\Im
	\left[
	r \rho \left\{
	r^2 + 2 r \cot ( \theta_i - \phi_i ) + \mathrm{cosec}^2 ( \theta_i - \phi_i )
	\right\}
	\right]
	\right|
	\leqq C v^2 \leqq C v ,
\]
\[
	\begin{array}{l}
	\displaystyle{
	\left|
	\Im \left[
	\sum_{ j \in J }
	\frac
	{ n_j \sin^2 ( \theta_i - \phi_i ) \cos ( \phi_j - \phi_i ) \sin ( \phi_j - \phi_i ) }
	{ \displaystyle{
	\sin ( \theta_i - \phi_j )
	\left\{
	\sin ( \theta_i - \phi_j ) - \rho \sin ( \phi_j - \phi_i ) \sin ( \theta_i - \phi_i )
	\right\}
	} }
	\right]
	\right|
	}
	\\
	\quad
	\leqq
	\displaystyle{
	C \sum_{ j \in J }
	\frac
	{ \left|
	\Im \left\{
	\sin ( \theta_i - \phi_j ) - \rho \sin ( \phi_j - \phi_i ) \sin ( \theta_i - \phi_i )
	\right\}
	\right|
	}
	{ \displaystyle{
	\left|
	\sin ( \theta_i - \phi_j )
	\left\{
	\sin ( \theta_i - \phi_j ) - \rho \sin ( \phi_j - \phi_i ) \sin ( \theta_i - \phi_i )
	\right\}
	\right|^2
	} }
	}
	\\
	\quad
	\leqq
	C | \rho_{{}_I } |
	.
	\end{array}
\]
Similarly it holds for $ \Im \hat F_4 $ that
\[
	\begin{array}{rl}
	| \Im \hat F_4 ( \hat r , \hat \rho ) |
	= & \zume
	| \Im F_4 ( r , \rho ) |
	\\
	\leqq & \zume
	\displaystyle{
	\left| \Im \left[
	\rho^2
	\left\{
	r^2 + 2 r \cot ( \theta_i - \phi_i ) + \mathrm{cosec}^2 ( \theta_i - \phi_i )
	\right\}
	\right]
	\right|
	}
	\\
	& \quad
	\displaystyle{
	\times
	\left|
	\Re \left[
	\sum_{ j \in J }
	\frac
	{ n_j \sin^2 ( \theta_i - \phi_i ) \sin^2 ( \phi_j - \phi_i ) \cos ( \theta_i - \phi_j ) }
	{ \displaystyle{
	\sin^2 ( \theta_i - \phi_j )
	\left\{
	\sin ( \theta_i - \phi_j ) - \rho \sin ( \phi_j - \phi_i ) \sin ( \theta_i - \phi_i )
	\right\}
	} }
	\right]
	\right|
	}
	\\
	& \quad
	\displaystyle{
	+ \,
	\left| \Re \left[
	\rho^2
	\left\{
	r^2 + 2 r \cot ( \theta_i - \phi_i ) + \mathrm{cosec}^2 ( \theta_i - \phi_i )
	\right\}
	\right]
	\right|
	}
	\\
	& \qquad
	\displaystyle{
	\times
	\left|
	\Im \left[
	\sum_{ j \in J }
	\frac
	{ n_j \sin^2 ( \theta_i - \phi_i ) \sin^2 ( \phi_j - \phi_i ) \cos ( \theta_i - \phi_j ) }
	{ \displaystyle{
	\sin^2 ( \theta_i - \phi_j )
	\left\{
	\sin ( \theta_i - \phi_j ) - \rho \sin ( \phi_j - \phi_i ) \sin ( \theta_i - \phi_i )
	\right\}
	} }
	\right]
	\right|
	}
	\\
	\leqq &\zume
	\displaystyle{
	C \left(
	\left| \Im \left( \rho^2 r^2 \right) \right|
	+
	\left| \Im \left( r \rho^2 \right) \right|
	+
	\left| \Im \left( \rho^2 \right) \right|
	\right)
	+
	C v^2 | \rho_{{}_I } |
	}
	\\
	\leqq & \zume
	\displaystyle{
	C v \left( | r_{{}_I } | + | \rho_{{}_I } | \right)
	}
	.
	\end{array}
\]
\qed
\par
We need the following lemma for the estimate of $ \Im \hat F_5 $.
\begin{lem}
Let $ B_R $ be the closed disc in $ \mathbb{C} $ with center $ O $ and with radius $ R $.
Assume that the function $ f $ is analytic on $ B_R $,
and that $ f|_{ B_R \cap \mathbb{R} } $ is real-valued.
Then there exists $ C > 0 $ such that
\[
	| \Im f(r) | \leqq C | r_{{}_I } |
\]
holds for $ r \in B_{ R/2 } $.
\label{Imf}
\end{lem}
\begin{proof}
We have
\[
	f(r) = \sum_{ k=0 }^\infty c_k r^k
	\quad \mbox{for} \quad r \in B_R ,
	\quad \mbox{and} \quad
	\sum_{ k=0 } | c_k | R^k < \infty .
\]
Since $ f|_{ B_R \cap \mathbb{R} } $ is real-valued,
so are $ c_k $'s.
Therefore,
we get
\[
	\Im f(r) = \sum_{ k=1 }^\infty c_k \Im \left( r^k \right) .
\]
Furthermore, we have
\[
	\left| \Im \left( r^k \right) \right|
	=
	\left| \Im \left( r_{{}_R} + \sqrt{-1} r_{{}_I } \right)^k \right|
	\leqq
	\sum_{ \ell=1 }^k {}_k C_\ell \left| r_{{}_R}^{ k - \ell } r_{{}_I}^\ell \right|
	\leqq
	\sum_{ \ell=0 }^k {}_k C_\ell | r |^{ k - 1 } | r_{{}_I} |
	=
	\frac { ( 2| r | )^k } R | r_{{}_I } | .
\]
Consequently,
\[
	\left|
	\sum_{k=1}^\infty c_k \Im \left( r^k \right)
	\right|
	\leqq
	\frac { | r_{{}_I } | } R \sum_{k=1}^\infty c_k ( 2 | r | )^k ,
\]
and the right-hand side converges for $ r \in B_{R/2} $.
\qed
\end{proof}
{\em Proof of Lemme \ref{ImFk} for $ \Im \hat F_5 $}.
\par\noindent
\par
The function $ f(r) = \left[ \left\{ \cot ( \theta_i - \phi_i ) + r \right\}^2 + 1 \right]^{ \frac 32 } $ satisfies the assumption of Lemma \ref{Imf}.
Taking $ V $ small,
we may assume $ \displaystyle{ | r_{{}_I} (v) | \leqq \frac R2 } $.
Hence,
we have
\[
	\left| \Im
	\left[ \left\{ \cot ( \theta_i - \phi_i ) + r(v) \right\}^2 + 1 \right]^{ \frac 32 }
	\right|
	\leqq C | r_{{}_I } | .
\]
Consequently,
it holds that
\[
	| \Im \hat F_5 ( \hat r , \hat \rho ) |
	=
	| \Im F_5 ( r ) |
	\leqq
	C v
	\left| \Im
	\left[ \left\{ \cot ( \theta_i - \phi_i ) + r(v) \right\}^2 + 1 \right]^{ \frac 32 }
	\right|
	\leqq
	C v | r_{{}_I } | .
\]
\qed
\par
We have now completed the proof of existence of a real-valued local solution to \pref{equationrrho1},
which also proves the following proposition.
\begin{prop}
Let $ H $ be continuous.
Then there exists a unique local solution $ \vecx $ to {\rm \pref{equationx}} and {\rm \pref{initialcondition-ii}}.
\label{prop4.2}
\end{prop}
\section{Appendix}
\subsection{The values of $ \theta_i $}
\begin{fact}
If $ A ( \theta_i ) = 0 $,
$ \phi_i < \theta_i < \phi_{ i+1 } $,
then
\begin{description}
\item[{\rm Type II:}]
	$ \theta_0 = \arctan\sqrt{ \frac { n_0 } { n_1 } } $,
\item[{\rm Type III:}]
	$ \theta_i = \frac 12 ( \phi_i + \phi_{ i+1 } ) $,
\item[{\rm Type IV:}]
	$ \theta_{\pm 1 } = - \frac 12 \arctan \sqrt{ \frac k \ell } $,
	$ \theta_0 = \frac 12 \arctan \sqrt{ \frac k \ell } $,
\item[{\rm Type V:}]
	$ \theta_i = \frac 12 ( \phi_i + \phi_{ i+1 } ) $.
\end{description}
\label{thetai}
\end{fact}
\begin{proof}
For Type II:
Since $ J = \{ 0 , 1 \} $,
$ \phi_j = \frac j2 \pi $,
we have
\[
	0
	=
	\sum_{ j \in J } n_j \cot ( \theta_0 - \phi_j )
	=
	n_0 \cot \theta_0 + n_1 \cot \left( \theta_0 - \frac \pi 2 \right)
	=
	n_0 \cot \theta_0 - n_1 \tan \theta_0 .
\]
Combining this with $ \theta_0 \in ( 0 , \frac \pi 2 ) $,
we get the assertion.
\par
Next we deal with Types III--V.
When $ \pm j \in J $,
we have $ \phi_{ -j } = - \phi_j $,
$ n_{ -j } = n_j $.
Therefore,
it holds that
\[
	\cot ( \theta_i - \phi_{-j} ) + \cot ( \theta_i - \phi_j )
	=
	\frac
	{ \sin 2 \theta_i }
	{ \sin^2 \theta_i \cos^2 \phi_j - \cos^2 \theta_i \sin^2 \phi_j }
	.
\]
Furthermore $ - \max J \not \in J $,
$ n_0 = n_{\max J} $,
$ \phi_0 = 0 $ and $ \phi_{\max J} = \frac \pi 2 $ for Types IV and V.
For these cases
\[
	\cot ( \theta_i - \phi_0 ) + \cot ( \theta_i - \phi_{\max J} )
	=
	2 \cot 2 \theta_i .
\]
\par
Using these relations we have
\[
	\begin{array}{rl}
	0
	= & \zume
	\displaystyle{
	\sum_{ j \in J } n_j \cot ( \theta_i - \phi_j )
	=
	n_0 \sum_{ j = -1 }^1 \cot \left( \theta_i - \frac j3 \pi \right)
	}
	\\
	= & \zume
	\displaystyle{
	n_0
	\left[
	\left\{
	\cot \left( \theta_i + \frac \pi 3 \right)
	+
	\cot \left( \theta_i - \frac \pi 3 \right)
	\right\}
	+
	\cot \theta_i
	\right]
	}
	\\
	= & \zume
	\displaystyle{
	\frac
	{ 3 n_0 \cos \theta_i \left( 2 \sin \theta_i - 1 \right) \left( 2 \sin \theta_i + 1 \right) }
	{ \sin \theta_i \left( \sin^2 \theta_i - 3 \cos^2 \theta_i \right) }
	}
	\end{array}
\]
for Type III;
\[
	\begin{array}{rl}
	0 = & \zume
	\displaystyle{
	\sum_{ j \in J } n_j \cot ( \theta_i - \phi_j )
	}
	\\
	= & \zume
	\displaystyle{
	\ell
	\left\{
	\cot \left( \theta_i + \frac \pi 4 \right) + \cot \left( \theta_i - \frac \pi 4 \right)
	\right\}
	+
	k \left( \cot \theta_i - \tan \theta_i \right)
	}
	\\
	= & \zume
	- 2 \ell \tan 2 \theta_i + 2 k \cot 2 \theta_i
	\end{array}
\]
for Type IV;
\[
	\begin{array}{rl}
	0 = & \zume
	\displaystyle{
	\sum_{ j \in J } n_j \cot ( \theta_i - \phi_j )
	=
	n_0 \sum_{ j=-2 }^3 \cot \left( \theta_i - \frac j6 \pi \right)
	}
	\\
	= & \zume
	\displaystyle{
	n_0
	\left[
	\left\{
	\cot \left( \theta_i + \frac \pi 3 \right)
	+
	\cot \left( \theta_i - \frac \pi 3 \right)
	\right\}
	+
	\left\{
	\cot \left( \theta_i + \frac \pi 6 \right)
	+
	\cot \left( \theta_i - \frac \pi 6 \right)
	\right\}
	\right.
	}
	\\
	& \qquad
	\displaystyle{
	\left.
	+ \,
	\left\{
	\cot \theta_i
	+
	\cot \left( \theta_i - \frac \pi 2 \right)
	\right\}
	\right]
	}
	\\
	= & \zume
	\displaystyle{
	\frac { 6 n_0 \cos 2 \theta_i \left( 1 - 2 \sin 2 \theta_i \right) \left( 1 + 2 \sin 2 \theta_i \right) }
	{ \sin 2 \theta_i \left( \sin^2 \theta_i - 3 \cos^2 \theta_i \right)
	\left( 3 \sin^2 \theta_i - \cos^2 \theta_i \right) }
	}
	\end{array}
\]
for Type V.
Taking $ \theta_i \in ( \phi_i , \phi_{ i+1 } ) $ into consideration,
we have the assertion.
\qed
\end{proof}
\begin{rem}
The result $ \theta_i = \frac 12 ( \phi_i + \phi_{ i+1 } ) $ for Types III and V is by virtue of the symmetry $ n_j \equiv n_0 $.
\end{rem}
\subsection{The derivation of \pref{equationrrho1}}
\par
We insert $ r = q - \cot ( \theta_i - \phi_i ) $ into \pref{diffeqq1} with $ u(0) = 0 $.
It is trivial that
\[
	\frac { dq } { dv }
	= \frac { dr } { dv } ,
\]
\[
	( n - 1 )
	\left( q^2 + 1 \right)^{ \frac 32 } \tilde H
	=
	( n - 1 )
	\left[ \left\{ \cot ( \theta_i - \phi_i ) + r \right\}^2 + 1 \right]^{ \frac 32 } 
	\tilde H
	=
	\frac { F_5 (r) } v .
\]
The summation of remainder terms is
\[
	\begin{array}{l}
	\displaystyle{
	- \frac { n_i \left( q^2 + 1 \right) q } v
	+
	\sum_{ j \ne i }
	\frac
	{ n_j \left( q^2 + 1 \right)
	\left\{ q \cos ( \phi_j - \phi_i ) + \sin ( \phi_j - \phi_i ) \right\}
	}
	{ \displaystyle{
	\sin ( \phi_j - \phi_i ) \int_0^v q( \eta ) \, d \eta
	- v \cos ( \phi_j - \phi_i )
	} }
	}
	\\
	\quad
	=
	\displaystyle{
	\left( q^2 + 1 \right)
	\sum_{ j \in J }
	\frac
	{ n_j
	\left\{ q \cos ( \phi_j - \phi_i ) + \sin ( \phi_j - \phi_i ) \right\}
	}
	{ \displaystyle{
	\sin ( \phi_j - \phi_i ) \int_0^v q( \eta ) \, d \eta
	- v \cos ( \phi_j - \phi_i )
	} }
	}
	\\
	\quad
	=
	\displaystyle{
	\left[
	\left\{ \cot ( \theta_i - \phi_i ) + r \right\}^2 + 1
	\right]
	}
	\\
	\quad \qquad
	\displaystyle{
	\times
	\sum_{ j \in J }
	\frac
	{ n_j \left[
	\left\{ \cot ( \theta_i - \phi_i ) + r \right\} \cos ( \phi_j - \phi_i )
	+ \sin ( \phi_j - \phi_i )
	\right] }
	{ \displaystyle{
	\sin ( \phi_j - \phi_i )
	\left\{ v \cot ( \theta_i - \phi_i ) + \int_0^v r( \eta ) \, d \eta \right\}
	- v \cos ( \phi_j - \phi_i )
	} }
	}
	\\
	\quad
	=
	\displaystyle{
	-
	\left\{
	\cot^2 ( \theta_i - \phi_i )
	+ 2 r \cot ( \theta_i - \phi_i )
	+ r^2 + 1
	\right\}
	}
	\\
	\quad \qquad
	\displaystyle{
	\times
	\sum_{ j \in J }
	\frac
	{ n_j \left\{
	\cos ( \theta_i - \phi_j ) + r \cos ( \phi_j - \phi_i ) \sin ( \theta_i - \phi_i )
	\right\} }
	{ \displaystyle{
	v \left\{
	\sin ( \theta_i - \phi_j )
	- \rho \sin ( \phi_j - \phi_i ) \sin ( \theta_i - \phi_i )
	\right\}
	} }
	}
	.
	\end{array}
\]
Using $ \displaystyle{ \sum_{ j \in J } n_j \cot ( \theta_i - \phi_j ) = 0 } $,
we have
\[
	\begin{array}{l}
	\displaystyle{
	-
	\left\{
	\cot^2 ( \theta_i - \phi_i )
	+ 2 r \cot ( \theta_i - \phi_i )
	+ r^2 + 1
	\right\}
	}
	\\
	\quad \qquad
	\displaystyle{
	\times
	\sum_{ j \in J }
	\frac
	{ n_j \left\{
	\cos ( \theta_i - \phi_j ) + r \cos ( \phi_j - \phi_i ) \sin ( \theta_i - \phi_i )
	\right\} }
	{ \displaystyle{
	v \left\{
	\sin ( \theta_i - \phi_j )
	- \rho \sin ( \phi_j - \phi_i ) \sin ( \theta_i - \phi_i )
	\right\}
	} }
	}
	\\
	\quad
	=
	\displaystyle{
	-
	\left\{
	r^2 + 2 r \cot ( \theta_i - \phi_i ) + \mathrm{cosec}^2 ( \theta_i - \phi_i )
	\right\}
	}
	\\
	\quad \qquad
	\displaystyle{
	\times
	\sum_{ j \in J }
	\frac { n_j } v
	\left\{
	\frac
	{ \cos ( \theta_i - \phi_j ) + r \cos ( \phi_j - \phi_i ) \sin ( \theta_i - \phi_i ) }
	{ \displaystyle{
	\sin ( \theta_i - \phi_j )
	- \rho \sin ( \phi_j - \phi_i ) \sin ( \theta_i - \phi_i )
	} }
	-
	\frac { \cos ( \theta_i - \phi_j ) } { \sin ( \theta_i - \phi_j ) }
	\right\}
	}
	\\
	\quad
	=
	\displaystyle{
	-
	\left\{
	r^2 + 2 r \cot ( \theta_i - \phi_i ) + \mathrm{cosec}^2 ( \theta_i - \phi_i )
	\right\}
	}
	\\
	\quad \qquad
	\displaystyle{
	\times
	\sum_{ j \in J }
	\frac
	{ \displaystyle{
	n_j \sin ( \theta_i - \phi_i )
	\left\{
	r \cos ( \phi_j - \phi_i ) \sin ( \theta_i - \phi_j )
	+
	\rho \sin ( \phi_j - \phi_i ) \cos ( \theta_i - \phi_j )
	\right\}
	} }
	{ \displaystyle{
	v \sin ( \theta_i - \phi_j )
	\left\{
	\sin ( \theta_i - \phi_j ) - \rho \sin ( \phi_j - \phi_i ) \sin ( \theta_i - \phi_i )
	\right\}
	} }
	}
	.
	\end{array}
\]
Extracting the linear parts with respect to $ r $ and $ \rho $,
we get
\[
	\begin{array}{l}
	\displaystyle{
	-
	\left\{
	r^2 + 2 r \cot ( \theta_i - \phi_i ) + \mathrm{cosec}^2 ( \theta_i - \phi_i )
	\right\}
	}
	\\
	\quad \qquad
	\displaystyle{
	\times
	\sum_{ j \in J }
	\frac
	{ \displaystyle{
	n_j \sin ( \theta_i - \phi_i )
	\left\{
	r \cos ( \phi_j - \phi_i ) \sin ( \theta_i - \phi_j )
	+
	\rho \sin ( \phi_j - \phi_i ) \cos ( \theta_i - \phi_j )
	\right\}
	} }
	{ \displaystyle{
	v \sin ( \theta_i - \phi_j )
	\left\{
	\sin ( \theta_i - \phi_j ) - \rho \sin ( \phi_j - \phi_i ) \sin ( \theta_i - \phi_i )
	\right\}
	} }
	}
	\\
	\quad
	=
	\displaystyle{
	-
	\left\{
	r^2 + 2 r \cot ( \theta_i - \phi_i ) + \mathrm{cosec}^2 ( \theta_i - \phi_i )
	\right\}
	}
	\\
	\quad \qquad
	\displaystyle{
	\times
	\sum_{ j \in J }
	\frac { n_j \sin ( \theta_i - \phi_i ) } v
	\left[
	\frac
	{ r \cos ( \phi_j - \phi_i ) }
	{ \sin ( \theta_i - \phi_j ) }
	+
	\frac { \rho \sin ( \phi_j - \phi_i) \cos ( \theta_i - \phi_j ) }
	{ \sin^2 ( \theta_i - \phi_j ) }
	\right.
	}
	\\
	\qquad \qquad
	\displaystyle{
	+ \,
	\frac
	{ \displaystyle{
	\left\{
	r \cos ( \phi_j - \phi_i ) \sin ( \theta_i - \phi_j )
	+
	\rho \sin ( \phi_j - \phi_i ) \cos ( \theta_i - \phi_j )
	\right\}
	} }
	{ \sin ( \theta_i - \phi_j ) }
	}
	\\
	\qquad \qquad \quad
	\displaystyle{
	\left.
	\times
	\left\{
	\frac 1
	{ \displaystyle{
	\sin ( \theta_i - \phi_j ) - \rho \sin ( \phi_j - \phi_i ) \sin ( \theta_i - \phi_i )
	} }
	-
	\frac 1
	{ \sin ( \theta_i - \phi_j ) }
	\right\}
	\right]
	}
	\\
	\quad
	=
	\displaystyle{
	- r
	\sum_{ j \in J }
	\frac { n_j \cos ( \phi_j - \phi_i ) }
	{ v \sin ( \theta_i - \phi_i ) \sin ( \theta_i - \phi_j ) }
	}
	\\
	\quad \qquad
	\displaystyle{
	- \,
	r^2 \left\{ r + 2 \cot ( \theta_i - \phi_i ) \right\}
	\sum_{ j \in J }
	\frac
	{ n_j \sin ( \theta_i - \phi_i ) \cos ( \phi_j - \phi_i ) }
	{ v \sin ( \theta_i - \phi_j ) }
	}
	\\
	\quad \qquad
	\displaystyle{
	- \,
	\sum_{ j \in J }
	\frac
	{ n_j \rho \sin ( \phi_j - \phi_i ) \cos ( \theta_i - \phi_j ) }
	{ v \sin ( \theta_i - \phi_i ) \sin^2 ( \theta_i - \phi_j ) }
	}
	\\
	\quad
	\qquad
	\displaystyle{
	- \,
	r \left\{ r + 2 \cot ( \theta_i - \phi_i ) \right\}
	\sum_{ j \in J }
	\frac { n_j \rho \sin ( \theta_i - \phi_i ) \sin ( \phi_j - \phi_i ) \cos ( \theta_i - \phi_j ) }
	{ v \sin^2 ( \theta_i - \phi_j ) }
	}
	\\
	\quad \qquad
	\displaystyle{
	- \,
	\left\{
	r^2 + 2 r \cot ( \theta_i - \phi_i ) + \mathrm{cosec}^2 ( \theta_i - \phi_i )
	\right\}
	}
	\\
	\qquad \qquad
	\displaystyle{
	\times
	\sum_{ j \in J }
	\frac
	{ n_j r \rho \sin^2 ( \theta_i - \phi_i ) \cos ( \phi_j - \phi_i ) \sin ( \phi_j - \phi_i ) }
	{ \displaystyle{
	v \sin ( \theta_i - \phi_j )
	\left\{
	\sin ( \theta_i - \phi_j ) - \rho \sin ( \phi_j - \phi_i ) \sin ( \theta_i - \phi_i )
	\right\}
	} }
	}
	\\
	\quad \qquad
	\displaystyle{
	- \,
	\left\{
	r^2 + 2 r \cot ( \theta_i - \phi_i ) + \mathrm{cosec}^2 ( \theta_i - \phi_i )
	\right\}
	}
	\\
	\qquad \qquad
	\displaystyle{
	\times
	\sum_{ j \in J }
	\frac
	{ n_j \rho^2 \sin^2 ( \theta_i - \phi_i ) \sin^2 ( \phi_j - \phi_i ) \cos ( \theta_i - \phi_j ) }
	{ \displaystyle{
	v \sin^2 ( \theta_i - \phi_j )
	\left\{
	\sin ( \theta_i - \phi_j ) - \rho \sin ( \phi_j - \phi_i ) \sin ( \theta_i - \phi_i )
	\right\}
	} }
	}
	.
	\end{array}
\]
By using $ \displaystyle{ \sum_{ j \in J } n_j \cot ( \theta_i - \phi_j ) = 0 } $ again,
the coefficient of the linear terms are simplified as follows:
\[
	\begin{array}{l}
	\displaystyle{
	\sum_{ j \in J }
	\frac { n_j \cos ( \phi_j - \phi_i ) }
	{ \sin ( \theta_i - \phi_i ) \sin ( \theta_i - \phi_j ) }
	}
	=
	\displaystyle{
	\sum_{ j \in J }
	\frac
	{ n_j \cos \left\{ ( \phi_j - \theta_i ) + ( \theta_i - \phi_i ) \right\} }
	{ \sin ( \theta_i - \phi_i ) \sin ( \theta_i - \phi_j ) }
	}
	\\
	\quad
	=
	\displaystyle{
	\sum_{ j \in J }
	\frac
	{ n_j
	\left\{
	\cos ( \phi_j - \theta_i ) \cos ( \theta_i - \phi_i )
	-
	\sin ( \phi_j - \theta_i ) \sin ( \theta_i - \phi_i )
	\right\} }
	{ \sin ( \theta_i - \phi_i ) \sin ( \theta_i - \phi_j ) }
	}
	\\
	\quad
	=
	\displaystyle{
	\cot ( \theta_i - \phi_i )
	\sum_{ j \in J }
	n_j \cot ( \theta_i - \phi_j )
	+
	\sum_{ j \in J } n_j
	=
	\sum_{ j \in J } n_j
	= n - 2
	}
	,
	\end{array}
\]
\[
	\begin{array}{l}
	\displaystyle{
	\sum_{ j \in J }
	\frac { n_j \sin ( \phi_j - \phi_i ) \cos ( \theta_i - \phi_j ) }
	{ \sin ( \theta_i - \phi_i ) \sin^2 ( \theta_i - \phi_j ) }
	}
	=
	\displaystyle{
	\sum_{ j \in J }
	\frac
	{ n_j \sin \left\{ ( \phi_j - \theta_i ) + ( \theta_i - \phi_i ) \right\} \cos ( \theta_i - \phi_j ) }
	{ \sin ( \theta_i - \phi_i ) \sin^2 ( \theta_i - \phi_j ) }
	}
	\\
	\quad
	=
	\displaystyle{
	\sum_{ j \in J }
	\frac
	{ n_j
	\left\{
	\sin ( \phi_j - \theta_i ) \cos ( \theta_i - \phi_i )
	+
	\cos ( \phi_j - \theta_i ) \sin ( \theta_i - \phi_i )
	\right\}
	\cos ( \theta_i - \phi_j ) }
	{ \sin ( \theta_i - \phi_i ) \sin^2 ( \theta_i - \phi_j ) }
	}
	\\
	\quad
	=
	\displaystyle{
	-
	\cot ( \theta_i - \phi_i )
	\sum_{ j \in J }
	n_j \cot ( \theta_i - \phi_j )
	+
	\sum_{ j \in J }
	n_j \cot^2 ( \theta_i - \phi_j )
	}
	\\
	\quad
	=
	\displaystyle{
	\sum_{ j \in J }
	n_j \cot^2 ( \theta_i - \phi_j )
	.
	}
	\end{array}
\]
Consequently,
we get \pref{equationrrho1} if
\[
	\sum_{ j \in J } n_j \cot^2 ( \theta_i - \phi_j ) = \gamma ( n - 2 ) ,
\]
which we shall prove next.
To do this,
we use Fact \ref{thetai}.
\par
Type II:
Since $ J = \{ 0 , 1 \} $,
$ \gamma = \sharp J - 1 = 1 $,
$ \phi_j = \frac j2 \pi $,
and $ \tan^2 \theta_0 = \frac { n_0 } { n_1 } $,
we get
\[
	\begin{array}{rl}
	\displaystyle{
	\sum_{ j \in J } n_j \cot^2 ( \theta_0 - \phi_j )
	}
	= & \zume
	\displaystyle{
	n_0 \cot^2 \theta_0 + n_1 \cot^2 \left( \theta_0 - \frac \pi 2 \right)
	=
	n_0 \cot^2 \theta_0 + n_1 \tan^2 \theta_0
	}
	\\
	= & \zume
	\displaystyle{
	n_0 \cdot \frac { n_1 } { n_0 }
	+
	n_1 \cdot \frac { n_0 } { n_1 }
	=
	n_1 + n_0
	=
	\sum_{ j \in J } n_j
	=
	\gamma ( n - 2 ) .
	}
	\end{array}
\]
\par
As in the proof of Fact \ref{thetai},
we have the following for Types III--V.
When $ \pm j \in J $,
we have $ \phi_{ -j } = - \phi_j $,
$ n_{ -j } = n_j $.
Therefore,
it holds that
\[
	\cot^2 ( \theta_i - \phi_{-j} ) + \cot^2 ( \theta_i - \phi_j )
	=
	\frac
	{
	2 \left( \sin^2 \theta_i \cos^2 \theta_i + \sin^2 \phi_j \cos^2 \phi_j \right)
	}
	{
	\left(
	\sin^2 \theta_i \cos^2 \phi_j - \cos^2 \theta_i \sin^2 \phi_j
	\right)^2
	}
	.
\]
Furthermore $ - \max J \not \in J $,
$ n_0 = n_{\max J} $,
$ \phi_0 = 0 $ and $ \phi_{\max J} = \frac \pi 2 $ for Types IV and V.
For these cases
\[
	\cot^2 ( \theta_i - \phi_0 ) + \cot^2 ( \theta_i - \phi_{\max J} )
	=
	\frac { 2 ( 2 - \sin^2 2 \theta_i ) } { \sin^2 2 \theta_i }
\]
\par
Type III:
Since $ J = \{ -1 , 0 , 1 \} $,
$ \gamma = \sharp J - 1 = 2 $,
$ \phi_j = \frac j3 \pi $,
$ n_j \equiv n_0 $,
and $ \sin^2 \theta_i = \frac 14 $,
$ \cos^2 \theta_i = \frac 34 $ for all $ i \in J $,
we have
\[
	\begin{array}{l}
	\displaystyle{
	\sum_{ j \in J } n_j \cot^2 ( \theta_i - \phi_j )
	=
	n_0
	\left\{
	\cot^2 ( \theta_i - \phi_{-1} + \cot^2 ( \theta_i - \phi_1 ) + \cot^2 \theta_i
	\right\}
	}
	\\
	\quad
	=
	\displaystyle{
	n_0 \left\{
	\frac
	{ 2 \left( \sin^2 \theta_i \cos^2 \theta_i + \sin^2 \phi_1 \cos^2 \phi_1 \right) }
	{ \left( \sin^2 \theta_i \cos^2 \phi_1 - \cos^2 \theta_i \sin^2 \phi_1 \right)^2 }
	+
	\frac { \cos^2 \theta_i } { \sin^2 \theta_i }
	\right\}
	}
	\\
	\quad
	=
	\displaystyle{
	n_0
	\left\{
	\frac { 2 \left( \frac 14 \cdot \frac 34 + \frac 34 \cdot \frac 14 \right) }
	{ \left( \frac 14 \cdot \frac 14 - \frac 34 \cdot \frac 34 \right)^2 }
	+
	\frac { \frac 34 } { \frac 14 }
	\right\}
	}
	\\
	\quad
	=
	\displaystyle{
	n_0
	\left(
	\frac { \frac 34 } { \frac 14 } + 3
	\right)
	=
	6 n_0
	=
	2 \sum_{ j \in J } n_j
	=
	\gamma ( n - 2 ) .
	}
	\end{array}
\]
\par
Type IV:
Since $ J = \{ - 1 , 0 , 1 , 2 \} $,
$ \gamma = \sharp J - 1 = 3 $,
$ \phi_j = \frac j4 \pi $,
$ n_{-1} = n_1 = \ell $,
$ n_0 = n_2 = k $,
and
\[
	\sin^2 2 \theta_i = \frac k { k + \ell } ,
	\quad
	\cos^2 2 \theta_i = \frac \ell { k + \ell }
\]
for all $ i \in J $,
we have
\[
	\begin{array}{l}
	\displaystyle{
	\sum_{ j \in J } n_j \cot^2 ( \theta_i - \phi_j )
	}
	\\
	\quad
	=
	\displaystyle{
	\ell \left\{
	\cot^2 \left( \theta_i + \frac \pi 4 \right) + \cot^2 \left( \theta_i - \frac \pi 4 \right)
	\right\}
	+
	k \left\{
	\cot^2 \theta_i + \cot^2 \left( \theta_i - \frac \pi 2 \right)
	\right\}
	}
	\\
	\quad
	=
	\displaystyle{
	\frac
	{ 2 \ell \left(
	\sin^2 \theta_i \cos^2 \theta_i + \sin^2 \frac \pi 4 \cos^2 \frac \pi 4
	\right) }
	{ \left(
	\sin^2 \theta_i \cos^2 \frac \pi 4 - \cos^2 \theta_i \sin^2 \frac \pi 4
	\right)^2
	}
	+
	\frac
	{ 2 k \left( 2 - \sin^2 2 \theta_i \right) } { \sin^2 2 \theta_i }
	}
	\\
	\quad
	=
	\displaystyle{
	\frac
	{ 2 \ell \left(
	\sin^2 \theta_i \cos^2 \theta_i + \frac 14
	\right) }
	{ \frac 14 \left( \sin^2 \theta_i - \cos^2 \theta_i \right)^2 }
	+
	\frac
	{ 2 k \left( 1 + \cos^2 2 \theta_i \right) } { \sin^2 2 \theta_i }
	}
	\\
	\quad
	=
	\displaystyle{
	\frac
	{ 2 \ell \left( 1 + \sin^2 2 \theta_i \right) } { \cos^2 2 \theta_i }
	+
	\frac
	{ 2 k \left( 1 + \cos^2 2 \theta_i \right) } { \sin^2 2 \theta_i }
	}
	\\
	\quad
	=
	\displaystyle{
	2 ( k + \ell + k ) + 2 ( k + \ell + \ell )
	=
	6 ( k + \ell )
	=
	3 \sum_{ j \in J } n_j
	= \gamma ( n - 2 ) .
	}
	\end{array}
\]
\par
Type V:
Since $ J = \{ -2 , -1 , 0 , 1 , 2 , 3 \} $,
and $ n_j \equiv n_0 $,
we have
\[
	\begin{array}{l}
	\displaystyle{
	\sum_{ j \in J } n_j \cot^2 ( \theta_i - \phi_j )
	}
	\\
	\quad
	=
	\displaystyle{
	n_0
	\left[
	\sum_{ j=1 }^2
	\left\{
	\cot^2 ( \theta_i - \phi_{-j} ) + \cot^2 ( \theta_i - \phi_j )
	\right\}
	+
	\cot^2 \theta_i + \cot^2 ( \theta_i - \phi_3 )
	\right]
	}
	\\
	\quad
	=
	\displaystyle{
	n_0
	\left\{
	\sum_{ j=1 }^2
	\frac
	{ 2 \left( \sin^2 \theta_i \cos^2 \theta_i + \sin^2 \phi_j \cos^2 \phi_j \right) }
	{ \left( \sin^2 \theta_i \cos^2 \phi_j - \cos^2 \theta_i \sin^2 \phi_j \right)^2 }
	+
	\frac
	{ 2 \left( 2 - \sin^2 2 \theta_i \right) } { \sin^2 2 \theta_i }
	\right\}
	}
	.
	\end{array}
\]
When $ i = - 2 $,
$ 1 $,
it holds that
\[
	\sin^2 \theta_i = \cos^2 \theta_i = \frac 12 ,
	\quad
	\sin^2 2 \theta_i = 1 .
\]
Using $ \gamma = \sharp J - 1 = 5 $ and $ \phi_j =\frac j6 \pi $,
we get
\[
	\begin{array}{rl}
	\displaystyle{
	\sum_{ j \in J } n_j \cot^2 ( \theta_i - \phi_i )
	}
	= & \zume
	\displaystyle{
	n_0
	\left\{
	\sum_{ j=1 }^2
	\frac { 2 \left( \frac 14 + \sin^2 \phi_j \cos^2 \phi_j \right) }
	{ \frac 14 \left( \cos^2 \phi_j - \sin^2 \phi_j \right)^2 }
	+ 2 ( 2 - 1 )
	\right\}
	}
	\\
	= & \zume
	\displaystyle{
	n_0
	\left\{
	\sum_{ j=1 }^2
	\frac { 2 \left( 2 + \sin^2 2 \phi_j \right) } { \cos^2 2 \phi_j }
	+ 2
	\right\}
	}
	\\
	= & \zume
	\displaystyle{
	n_0
	\left\{
	\frac { 2 \left( 1 + \sin^2 \frac \pi 3 \right) } { \cos^2 \frac \pi 3 }
	+
	\frac { 2 \left( 1 + \sin^2 \frac 23 \pi \right) } { \cos^2 \frac 23 \pi }
	+
	2
	\right\}
	}
	\\
	= & \zume
	\displaystyle{
	n_0
	\left\{
	\frac { 2 \left( 1 + \frac 34 \right) } { \frac 14 }
	+
	\frac { 2 \left( 1 + \frac 34 \right) } { \frac 14 }
	+
	2
	\right\}
	}
	\\
	= & \zume
	\displaystyle{
	30 n_0
	= 5 \sum_{ j \in J } n_j
	= \gamma ( n - 2 ) .
	}
	\end{array}
\]
When $ i = -1 $,
$ 0 $,
$ 2 $,
it holds that
\[
	\sin^2 2 \theta_i = \frac 14 ,
	\quad
	\sin^2 \theta_i = \frac { 2 - \sqrt 3 } 4 ,
	\quad
	\cos^2 \theta_i = \frac { 2 + \sqrt 3 } 4 .
\]
Therefore,
we obtain
\[
	\begin{array}{l}
	\displaystyle{
	\sum_{ j \in J } n_j \cot^2 ( \theta_i - \phi_j )
	}
	\\
	\quad
	=
	\displaystyle{
	n_0
	\left[
	\sum_{ j=1 }^2
	\frac
	{ 2 \left(
	\frac { 2 - \sqrt 3 } 4 \cdot \frac { 2 + \sqrt 3 } 4
	+
	\sin^2 \phi_j \cos^2 \phi_j
	\right) }
	{ \frac 1 {16} \left\{
	\left( 2 - \sqrt 3 \right) \cos^2 \phi_j
	-
	\left( 2 + \sqrt 3 \right) \sin^2 \phi_j
	\right\}^2 }
	+
	\frac { 2 \left( 2 - \frac 14 \right) } { \frac 14 }
	\right]
	}
	\\
	\quad
	=
	\displaystyle{
	n_0
	\left\{
	\sum_{ j=1 }^2
	\frac
	{ 2 \left( 1 + 4 \sin^2 2 \phi_j \right) }
	{ \left( 2 \cos 2 \phi_j - \sqrt 3 \right)^2 }
	+
	\frac { 2 \cdot \frac 74 } { \frac 14 }
	\right\}
	}
	\\
	\quad
	=
	\displaystyle{
	n_0
	\left\{
	\frac { 2 \left( 1 + 4 \sin^2 \frac \pi 3 \right) }
	{ \left( 2 \cos \frac \pi 3 - \sqrt 3 \right)^2 }
	+
	\frac { 2 \left( 1 + 4 \sin^2 \frac 23 \pi \right) }
	{ \left( 2 \cos \frac 23 \pi - \sqrt 3 \right)^2 }
	+
	14
	\right\}
	}
	\\
	\quad
	=
	\displaystyle{
	n_0
	\left\{
	\frac { 2 ( 1 + 3 ) } { \left( 1 - \sqrt 3 \right)^2 }
	+
	\frac { 2 ( 1 + 3 ) } { \left( - 1 - \sqrt 3 \right)^2 }
	+
	14
	\right\}
	}
	\\
	\quad
	=
	\displaystyle{
	n_0
	\left(
	\frac 4 { 2 - \sqrt 3 } + \frac 4 { 2 + \sqrt 3 } + 14
	\right)
	}
	\\
	\quad
	=
	\displaystyle{
	n_0
	\left[
	4 \left\{ \left( 2 + \sqrt 3 \right) + \left( 2 - \sqrt 3 \right) \right\}
	+ 14
	\right]
	}
	\\
	\quad
	=
	\displaystyle{
	30 n_0
	= 5 \sum_{ j \in J } n_j
	= \gamma ( n - 2 )
	}
	\end{array}
\]

\end{document}